\documentclass[a4paper]{siamart220329}

\usepackage{damacros}

\newcommand{\norm}[1]{\left \| #1 \right \|}
\newcommand{\epsi}{\varepsilon}

\newcommand{\linf}[1]{\underset{#1\to\infty}{\underline{\lim}}}

\newcommand{\lsup}[1]{\underset{#1\to\infty}{\overline{\lim}}}
\renewcommand{\epsilon}{\varepsilon}
\renewcommand{\phi}{\varphi}

\newcommand{\rev}[1]{{\color{black} #1}} 

\title{Neural Fields and Noise-Induced Patterns in Neurons on Large Disordered Networks}

\author{%
  Daniele Avitabile%
  \thanks{%
    Amsterdam Centre for Dynamics and Computation,
    Vrije Universiteit Amsterdam,
    Department of Mathematics,
    Faculteit der Exacte Wetenschappen,
    De Boelelaan 1081a,
    1081 HV Amsterdam, The Netherlands.
  \protect
    MathNeuro Team,
    Inria branch of the University of Montpellier,
    860 rue Saint-Priest
    34095 Montpellier Cedex 5
    France.
  \protect
  (\email{d.avitabile@vu.nl}, \url{www.danieleavitabile.com}, \url{www.amsterdam-dynamics.nl}).
  }
  \and
  James MacLaurin
  \thanks{Department of Mathematical Sciences. 
    New Jersey Institute of Technology, \email{james.n.maclaurin@njit.edu}
  }
}

\graphicspath{{./Figures/}}

\begin{document}

\maketitle

\begin{abstract}
  We study pattern formation in a class of high-dimensional neural networks posed on
  random graphs and subject to spatio-temporal stochastic forcing. Under generic
  conditions on coupling and nodal dynamics, we prove that the network admits a
  rigorous mean-field limit, resembling a Wilson-Cowan neural field equation. The
  state variables of the limiting systems are the mean and variance of neuronal
  activity. We select networks whose mean-field equations are tractable and we
  perform a bifurcation analysis using as control parameter the diffusivity strength
  of the afferent white noise on each neuron. We find conditions for Turing-like
  bifurcations in a system where the cortex is modelled as a ring, and we produce
  numerical evidence of noise-induced spiral waves in models with a two-dimensional
  cortex. We provide numerical evidence that solutions of the finite-size network converge
  weakly to solutions of the mean-field model. Finally, we prove a Large Deviation
  Principle, which provides a means of assessing the likelihood of deviations from
  the mean-field equations induced by finite-size effects.
\end{abstract}

\section{Introduction}

A key aspect of network dynamics, and in particular in mathematical neuroscience,
is to understand how nodal (neuronal) dynamics, in combination with network (synaptic)
structure, shapes patterns in spatially-extended networks. In deterministic systems
with local interactions, a prominent mechanism for the onset of patterns from a
homogeneous, quiescent state is the one proposed by Turing in the context of
morphogenesis~\cite{Turing1952}. 

Turing's paradigm has been adopted and extended in numerous ways to explain pattern
formation in natural systems (see
\cite{krauseIntroductionRecentProgress2021,krauseModernPerspectivesNearequilibrium2021}
for recent perspectives on Turing's work and influence). 
Turing bifurcations are typically studied in reaction-diffusion systems with
local interactions, but instabilities with an identical bifurcation structure are
also found in cortical models, in which neurons are coupled non-locally and without
diffusion
\cite{ermentroutMathematicalFoundationsNeuroscience2010,bressloffWavesNeuralMedia2014,
coombesNeurodynamicsAppliedMathematics2023}. The mathematical treatment of these
Turing-like bifurcations requires minor modifications, in order to deal with
non-locality~\cite{ermentroutSpatiotemporalPatternFormation2014}.

In the present paper we investigate the onset of patterns in randomly-connected
neurons whose dynamics is also subject to noise. This research question is long
standing: 
two early influential papers by Othmer and Scriven
\cite{Othmer1971, Othmer1974} conjectured that large random networks could support
pattern-forming instabilities in a mechanism analogous to the Turing instabilities
for continuum PDEs. 
More recently, several groups worked theoretically, and experimentally on the impact
of white noise on pattern formation
\cite{van1994noise,zhou2001array,lindner2004effects,forgoston2018primer,carrillo2023noise,mckane2014stochastic,biancalani2017giant},
and it has been proposed that noise could explain why patterns are so prevalent in
nature, despite the fact that often they can only be proven to exist in confined, and
sometimes narrow, parametric regimes \cite{maini2012turing}.

It has been conjectured that noise (both white noise temporal
fluctuations, and disorder in the connectivity structure) is a key mechanism behind
oscillations and rhythms in the brain \cite{wang2010neurophysiological,
lindner2004effects}. On the front of stochastic dynamics, early work by Scheutzow
\cite{scheutzow1985noise} demonstrated that mean-field systems can demonstrated
oscillatory activity. More recently, Lu\c{c}on and Poquet have demonstrated that
white noise can excite oscillations in a high-dimensional all-to-all network of
coupled excitable neurons \cite{luccon2020emergence,luccon2024existence}. On the
front of random connections, recent work by Bramburger and Holzer
\cite{bramburger2023pattern} tackled Turing bifurcations in the Swift-Hohenburg
model on random graphs. Another recent work by Carrillo, Roux and Solem
\cite{carrillo2023noise,carrillo2023well} has proved the existence of noise-induced
bifurcations in neural field equations.

In the present paper we study noise-induced bifurcations and noise-induced
transitions in
spatially-extended networks of $n$ randomly-coupled, stochastically-forced neurons,
which admit a tractable $n \to \infty $ limit, the so-called \textit{mean-field limit}. The
mean-field limit, in our case, pertains to the average and variance of the local neuronal
activity, and resembles a \textit{neural field equation}.
Neural fields are heuristic models for the coarse-grained activity of
spatially-extended neuronal network, which have been extensively employed by
mathematical neuroscientists to predict a rich range of patterns, waves and other
coherent
structures~\cite{wilsonExcitatoryInhibitoryInteractions1972,amariHomogeneousNetsNeuronlike1975,ermentroutMathematicalFoundationsNeuroscience2010,
bressloffWavesNeuralMedia2014,coombesNeuralFields2014,coombesNeurodynamicsAppliedMathematics2023,cook2022neural}.

An important open question concerns the derivation of neural-field equations as mean-field
limits of microscopic neuronal models~\cite{bressloff2011spatiotemporal}. In this
respect, the classical mean-field theory developed by McKean
\cite{mckean1956elementary}, Sznitman \cite{sznitman1991topics}, and Tanaka
\cite{tanaka1984limit} can be adapted to determine the large population-density limit
of a network of neurons with a spatially-varying structure.  Touboul
\cite{touboul2014spatially,TOUBOUL20121223} has derived neural-field equations from
particle models, with a connectivity that is a function of the spatial positions of
the neurons. He found a range of interesting patterns and bifurcations, including
standing waves and hopf bifurcations.

In this line of work~\cite{Delattre2016,Lucon2020Quenched}, neurons are thought of as
interacting particles with heterogeneous coupling, and population-density evolution
equations of McKean-Vlasov type are found in the mean-field limit. While
such evolution equations can be derived in fairly general setups, their complexity
often hinders further analysis, hence addressing questions concerning pattern
formation, bifurcation, and noise-induced transitions remain often prohibitively
difficult.

Further, the limiting equations are not necessarily similar to the classical
Wilson-Cowan neural field equations \cite{wilsonExcitatoryInhibitoryInteractions1972}, owing to differential operators featuring in the
Fokker-Planck equation. From an intuitive standpoint, classical neural field equations are
heuristic and reduce the neural activity at a spatial position to a finite
set of variables; on the other hand, Fokker-Planck equations are rigorously derived,
but contain the probabilistic distribution of all possible states at a particular
spatial location.
\begin{figure}
  \centering
  \includegraphics{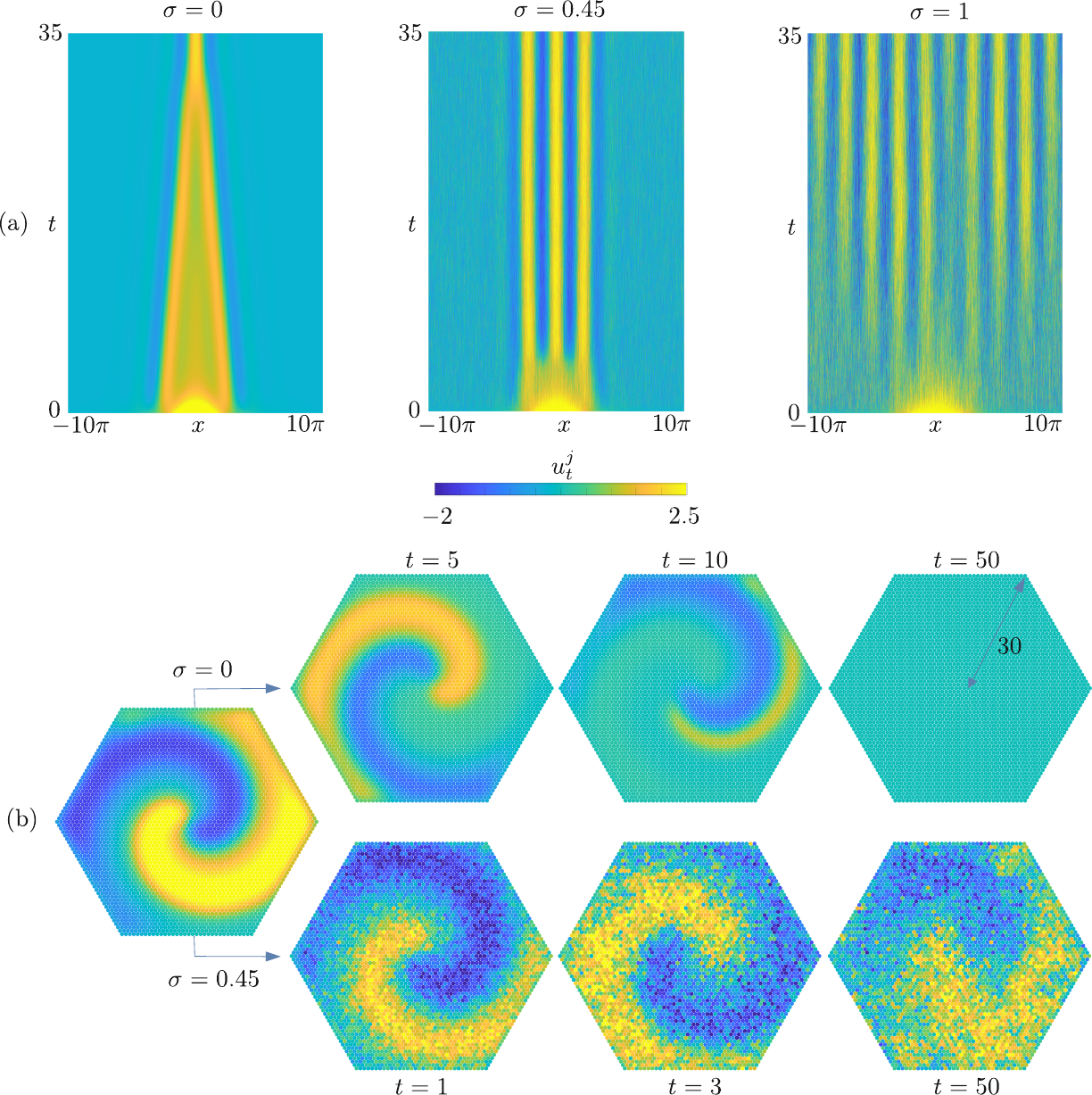}
  \caption{Noise-induced spatiotemporal patterns. (a) A rate model with $2^{13}$
    neurons, uniformly
    distributed on a ring of width $20\pi$ with all-to-all connections, displays a
    stationary localised equilibrium with 1 bump centred at $x=0$ in the
    absence of noise; when parameters and initial conditions are fixed, and noise is
    switched on, the model displays a different localised solution, with
    3 bumps at the core (noise intensity $\sigma = 0.45$), and a non-localised
    stationary state ($\sigma =1$). (b) A rate model with 3367 sparsely-coupled neurons
    distributed on a hexagon with diameter 60 supports a persistent rotating wave when
    noise is present ($\sigma = 0.45$,
    \href{https://figshare.com/articles/media/Animations_for_the_paper_Neural_fields_and_Noise-Induced_Patterns_in_Neurons_on_Large_Disordered_Networks_/26763037?file=48618040}{click here for an animation}), but the wave disappears when noise is absent
    (\href{https://figshare.com/articles/media/Animations_for_the_paper_Neural_fields_and_Noise-Induced_Patterns_in_Neurons_on_Large_Disordered_Networks_/26763037?file=48618037}{click for animation}).
    Parameters: for (a) we have used the model detailed in \cref{ssec:convergence}
    with $L =1$, $l =10 \pi$, $I(t,x) \equiv 0$, $G(x,t) \equiv \sigma$, $B = 0.4$, 
    $\alpha = 10$, $\theta = 0.4$, $m_0(x)= 5/\cosh(0.25 x)$; for (b) we use the model
    described in \cref{ssec:spiral} with $L_{11}=1 $, $L_{12}= 2$, $L_{21}= -0.44$,
    $L_{22}= 0.2 $ $I(t,x) \equiv 0$, $G(x,t) \equiv \sigma$, $\nu=1.5$, $\alpha = 10$,
    $\theta = 0.6$, with initial condition derived from a self-sustained spiral wave
    obtained in \cite[Tutorial 2, and accompanying data]{avitabile2024}.
  }
  \label{fig:patterns}
\end{figure}
Our aim is to determine limiting equations that are as tractable as classical neural
field equations, while being rigorously derived as the large $n$
limit from a system of interacting neurons. One strategy to obtain limiting
neural-field equations (both formally and rigorously) is to simplify the
microscopic dynamics of neurons, so as to describe neuronal spikes as instantaneous
events. Pattern formation in formal mean-field limits of integrate-and-fire spiking
neurons have been recently studied, revealing that stationary patterns in the mean
field may be a form of localised spatiotemporal
chaos~\cite{avitabileBumpAttractorsWaves2023}. Further, neural-field equations
have been found as rigorous limits of Hawkes Processes, in which the probability of a
neuronal spike at a particular time is the average of the activity of other neurons.
For this class of model, one can indeed achieve a limiting equation that resembles
the classical Wilson-Cowan neural field equation, as long as the memory kernel is
exponential~\cite{chevallier2019mean,agathe2022multivariate}. The state variable
in the neural fields is the local spiking intensity function.

No large brain network is `all-to-all'. In fact typically most brain networks are
sparse, insofar as the number of neurons connected to any particular neuron is much
smaller than the total number of neurons in the network. We thus consider a general
setup where the connectivity graph is inhomogeneous, and we make certain assumptions
on the average level of connectivity over the network. In the last section, we prove
that these assumptions are satisfied if the connections were to be sampled
independently from a distribution (with the probability of a connection depending on
the spatial location of the two neurons). Indeed there has been considerable effort
in recent years towards understanding dynamical systems on high dimensional random
graphs
\cite{oliveira2019interacting,delattre2016note,delgadino2021diffusive,bick2024dynamical,braunsteins2023sample,Remco2024}.

We work with McKean-Vlasov limits of spatially-extended neural networks of $n$
neurons of the form
\begin{equation}\label{eq:particleIntro}
   du^{j}_t = \Bigl(-L u^j_t + \frac{1}{n\phi_n} \sum_{k=1}^{n} K^{jk} f(u^k_t) + I^j_t\Bigr) dt
    + G^j_t dW^j_t, \quad j = 1,\ldots,n, \quad t \in \RSet_{\geq 0},
\end{equation}
driven by independent $\mathbb{R}^q$-valued Brownian Motions $\lbrace W^j_t \rbrace_{j}$.
Our choice for microscopic dynamics is thus a rate model: the
variable $u^j_t$ is a vector containing state variables of neuron at position $x^j_n$ and
time $t$; in addition to membranal voltage, the vector may contain other neural
variables such as synaptic or recovery variables; in addition, the form
presented above is also a compact way of describing the dynamics of multiple co-located
neuronal populations, whose variables at positions $j$ are stacked into the vector
$u^j_t$. Neurons are coupled through their firing rates $f$ via the
excitatory-inhibitory, and potentially sparse synaptic connections $K^{jk}$, with
strengths that scale with $n$ via the factor $n \phi_n$, whose asymptotics as $n
\to \infty $ will be discussed later. The network is subject to a
deterministic external input, as well as additive noise, both of which have a
spatio-temporal profile. In addition to stochastic forcing, the weights $K$ are
randomly distributed.

Crucially, we have kept the local neuronal dynamics linear, via the matrix $L$. Our
mathematical derivation of the mean-field limit can be modified to account for
nonlinear local interactions, but we refrain from this here: we rather leverage this
choice to determine a limiting mean-field equation that resembles the Wilson-Cowan
neural field equation. 

In \cref{fig:patterns} we show two examples of noise-induced patterns in networks of
type \cref{eq:particleIntro}. In the first one neurons are distributed on a ring of
width $20\pi$ and subject to a constant noise with intensity $G^j_t \equiv \sigma$.
In the absence of noise, the network produces a localised stationary state with one
bump at the core; upon increasing $\sigma$, while keeping all other elements of the
model untouched, the system displays a localised state with 3 bumps at the core,
or a nonlocalised state (see \cref{fig:patterns}(a)). The latter are not steady
states in the particle system \cref{eq:particleIntro}, because of the presence of
stochastic forcing. Further, if we fix a spatial location $j$ and time $t$, and send
$n \to \infty$, the neuronal coordinates $u^j_t$ in any $\epsilon$-ball $B_{\epsilon}(j)$ about $x^j_n$ \textit{do not converge} towards a
limit: there is indeed no pointwise convergence of the spatial profiles, and indeed the
higher $n$, the wider the fluctuations exhibited by the random variable $u^j_t$ (that is,  $\lim_{n\to\infty}\sup \big\lbrace \| u^j_t \| : x^j_n \in B_{\epsilon}(j) \big\rbrace = \infty$).

The expectation and variance of $u^j_t$, however, converge to a limit as $n \to
\infty$, and the evolution equation for their limits resemble Wilson-Cowan neural
field equations, in which the noise intensity $\sigma$ enters as a control parameter.
The transitions seen in \cref{fig:patterns}(a) can thus be explained in the mean
field limit by means of standard bifurcation theory, and numerical bifurcation
analysis. 

In \cref{fig:patterns}(b), we present an example in which neurons are distributed on
a hexagon (note that the cortex here is a discrete lattice with around 3400 neurons,
as visible from the image pixels and the
\href{https://figshare.com/articles/media/Animations_for_the_paper_Neural_fields_and_Noise-Induced_Patterns_in_Neurons_on_Large_Disordered_Networks_/26763037/1}{accompanying
animations}). The vector $u^j_t$ features a recovery variable,
and the setup of this numerical experiment is inspired by a seminal paper by Huang
et al.~\cite{huangSpiralWavesDisinhibited2004a}, in which spiral waves were observed
in neocortical slices, and simulated with a neural field with a recovery variable.
In \cref{fig:patterns}(b) we propose a simulation of a system of type
\cref{eq:particleIntro}: sufficiently large noise levels are necessary to support a
rotating wave in the system, which decays towards the inactive state in the noiseless
case.

The main results of the paper can be summarised as follows:
\begin{enumerate}
  \item We provide conditions on \cref{eq:particleIntro} which guarantee that the
    empirical measure associated to its particles converges as $n \to \infty $ to a
    unique measure. In addition to additive noise with spatio-temporal varying
    intensity, we work with quenched values of the random connections. 
  \item We prove that the marginal of the asymptotic measure is Gaussian at any
    spatial point, and any time. Further, we find that mean and variance of the
    law, which describe uniquely the marginal, evolve according to a system of
    neural-field equations.
  \item Using a concrete example in which the mean field can be easily computed, we
    illustrate how noise-induced Turing-like bifurcations arise in the mean-field
    model: we use the noise intensity as the bifurcation parameter and show how a stable
    homogeneous state becomes unstable to periodic perturbations of predictable
    wavelength as the noise intensity increases.
  \item We define an appropriate mode of convergence (weak convergence) for the
    random variables $u^j_t$ to the mean-field variables, and provide numerical
    evidence that this mode of convergence is attained in numerical simulations.
  \item We prove a Large Deviation Principle for the system.  
\end{enumerate}

The paper is structured as follows: we present notation in \cref{sec:notation}, and
introduce the particle model, our standing assumptions, and our main results in
\cref{Section Particle Model}; \cref{sec:Turing} presents Turing-like bifurcations in
a ring model, while \cref{sec:numerics} discusses numerical experiments; we present
mathematical proofs in \cref{sec:proof}, and we conclude in \cref{Section
Conclusion}. Extra details for some of the proofs are provided in the supplementary
materials (\cref{sec:supplementary}). 

\section{Notation}\label{sec:notation}
For any Polish Space $X$, we let $\mathcal{P}(X)$ denote the set
of all probability measures over $X$. We always endow
$\mathcal{P}(X)$ with the topology of weak convergence: this is the unique
topology such that $\mu_i \to \nu$ if and only if for any $g \in BC(X,\RSet)$, the
set of bounded continuous functions on $X$ to $\RSet$, it holds
$\mathbb{E}^{\mu_i}[g] \to \mathbb{E}^{\nu}[g]$, where for $\rho \in \calP(X)$ we
define
\[
  \mean^\rho [g] = \int_{X} g(x) \,d\rho(x).
\]
For $u \in C([0,T],\mathbb{R}^q)$, define the supremum norm
\begin{equation}\label{eq:normT}
  \| u \|_T = \sup_{t \in [0,T]} \sup_{1\leq \alpha \leq q}| u_{\alpha,t} |.
\end{equation}
More generally, for fixed compact $J \subset \RSet$ and Banach space $X$, we denote by $C(J,X)$ the
space of continuous functions on $J$ to $X$ with norm 
\[
\| u \|_{C(J, X)} = \max_{t \in J} \| u(t) \|_{X}.
\]
 
For an integer $n \in \NSet$ we set $\NSet_n = \{ 1, \ldots,
n\}$, and we denote by $\RSet^{n \times n}$ the space of real-valued $n$-by-$n$
matrices, endowed with the Frobenius norm $\| \blank \|_F$.

\section{Particle Model} \label{Section Particle Model}

We introduce in this section a spatially-extended model of
interacting neurons, subject to stochastic forcing, which we refer to as the
\textit{particle model}. The model features $q$ populations of neurons, each
containing $n$ neurons, for a total of $nq$ neurons. A straightforward adaptation of
the model covers the case of population containing a variable number of neurons so
that population $\alpha$ contains $n_\alpha$ neurons, for a total of $\sum_{\alpha =
1}^q n_\alpha$ neurons. This choice would make notation heavier without modifying
substantially our proofs, and we don't present it here. In passing we note that,
without any modification, the $q$-population model presented below covers also the
case of a single population of $n$ neurons, each described by $q$ variables.

The state of the system is described by a stochastic
process $\{ u_t \colon t \in [0,T]\}$ with values
in $\RSet^{nq}$. It is convenient to expose the components of $u_t$, hence we set
some notation to that purpose. Without loss of generality, we assume components to be
ordered lexicographically and use the notation
\[
  u_{\alpha,t}^j := (u_t)_{(j-1)q +\alpha}
  \in \RSet , \qquad (j,\alpha,t) \in \NSet_n \times \NSet_q \times [0,T]
\]
to denote the state variable of neuron $j \in \NSet_n$
in population $\alpha \in \NSet_q$ (neuron $(j,\alpha)$) at time $t \in
\RSet_{ \geq 0}$. Alternatively, in single population models in which each neuron
configuration has multiple components, $u_{\alpha,t}^j$ represents the $\alpha$th component
of the state variable of neuron $j$ at time $t$. With a little abuse of
notation, we will use interchangeably $u^j_{\alpha,t}$ and $(u_t)^j_\alpha$. As we
will see below, neurons are spatially distributed, with position $x^j$,
and it is convenient to introduce a symbol for all state variables located at the
$j$th node 
\[
  u^j_t := \{ u^j_{\alpha,t} \}_{j = 1}^q \in \RSet^q \qquad (j,t) \in \RSet^n \times
  [0,T].
\]

The particle model is a system of $nq$ It\^o Stochastic Ordinary Differential Equations of
the form
\begin{equation}\label{eq:particleModel}
  \begin{aligned}
    & du^j_t = \Bigl(-L u^j_t + \frac{1}{n \phi_n} \sum_{k=1}^{n} K^{jk} f(u^j_t) + I^j_t\Bigr) dt + G^j_t dW^j_t, 
    && (j,t) \in \NSet_n \times [0,T], \\ 
    & u^j_0 = z^j,
    && j \in \NSet_n
  \end{aligned}
\end{equation}
in which $L$ and $\{ K^{jk} \}_{j,k}$ are $q$-by-$q$ matrices representing the local
and nonlocal interaction coupling, respectively; $\phi_n$ is a coefficient scaling with
$n$, the function $f \colon \RSet^q \to \RSet^q$ models neuronal firing rates, the
vectors $\{ I^j_t \}_{j} \subset \RSet^q$ model deterministic external inputs at time
$t$, the $q$-by-$q$ matrices $\{ G^j_t\}_{j}$ model the intensity of
the noise (the diffusion term) at time $t$, and $\{ z^j \}_{j} \subset \RSet^q$ are the initial
conditions. In the remainder of this section we shall present all elements of the
model, so as to set the SODE on precise ground, and present our working assumptions. 

\subsection{Geometry}
We begin by making assumptions on the geometry of the cortex. Neurons are spatially
distributed on a cortex $D$ for which we make the following standing assumption:
\begin{hypothesis}[Cortical domain]\label{hyp:domain}
  The cortex $D$ is a compact domain in
  $\RSet^d$, for $d \in \NSet$, and there exists a metric $d(\blank,\blank)$ which
  defines a topology on $D$.  
\end{hypothesis}

In applications, the choice $D = \SSet^1$ or $D = \SSet^2$ is sometimes convenient,
but we are not restricted to this choice in the present paper. Further,
neuron $(j,\alpha)$ occupies position $x^j_n \in D$ for any $\alpha \in
\NSet_q$. In passing, note that we will omit the dependence on $n$ for notational
simplicity. Since $x^j$ is independent of $\alpha$, the $q$ populations are
co-located; this choice is also appropriate for a model with a single neuronal
population of neurons described by $q$ variables.

Neurons are distributed deterministically in $D$, but we are concerned with the
limiting behaviour of the model when $n \to \infty$, for which we require the
following assumption.
\begin{hypothesis}[Spatial distribution of neurons] \label{hyp:spatialDistributioninitial}
  The empirical measure
  \[
    \hat \mu^n = \frac{1}{n} \sum_{j \in \NSet_n} \delta_{x^j_n} 
    \in \calP(D)
  \]
  satisfies $\hat \mu^n \to \ell$ as $n \to \infty $ weakly in $\calP(D)$, where $\ell$ is
  the uniform measure on $D$.
\end{hypothesis}
\begin{remark}
  Although the distribution of points is deterministic, the empirical measure $\hat \mu^n(x)$ and its limit $\ell$ are still well-defined probability measures. This
  implies $\ell(D) = 1$ which differs from the Lebesgue measure of $D$ as one may
  expect initially. An appropriate scaling that accounts for the Lebesgue measure of
  $D$ will be introduced in due course.
\end{remark}

\subsection{Local and Nonlocal connectivity}\label{Section Connectivity}  The particle system
\cref{eq:particleModel} features both local and global neuronal interactions. The
local interaction is given by the \textit{deterministic, real-valued
matrix} $L \in \RSet^{q \times q}$ 
hence a local (possibly self-) coupling is in place between neuron $(j,\alpha)$ and
$(j,\beta)$. Note that, for simplicity, we have taken the matrix $L$ to be spatially
homogeneous, hence independent of $j$.

\begin{remark} It is possible to derive mean-field limits for particle systems
  with nonlinear local dynamics \cite{sznitman1991topics}. In the present paper we aim to
  illustrate the phenomenon of noise-induced Turing-like patterns using a relatively
  simple setup, which favours mathematical tractability of the mean field over
  generality. In this spirit, we restrict our attention to linear local
  cross-population coupling between neuron $(j,\alpha)$ and $(k,\beta)$. We note that
  nonlinear effects are still present, through the mean-field coupling term discussed
  below.
\end{remark}

The connection from neuron $(k,\beta)$ to $(j,\alpha)$ 
is modelled via matrices $\{ K^{jk} \}_{jk}$ 
acting on vectors $v \in \RSet^q$ as
\[
(K^{jk}v)_\alpha = \sum_{\beta \in \NSet_q} 
      K^{jk}_{\alpha\beta} v_{\beta},
      \qquad (j,k) \in \NSet_n.
\]
The connectivity may be sparse or dense. The level of sparseness is indicated by a
parameter $\phi_n$: this is such that the typical number of edges afferent on a
typical node is $n\phi_n$. The synaptic connectivity can be excitatory and
inhibitory, symmetric or asymmetric, and need not be distance dependent. We require
the following hypothesis to hold, which is similar to a Graphon Assumption employed in
many papers on high-dimensional deterministic ODEs on sparse networks
\cite{bramburger2023pattern}.
\begin{hypothesis}[Synaptic connections] \label{hyp:connections}
  The synaptic connection strength from neuron
  $(k,\beta)$ to $(j,\alpha)$ is given by $(n\phi_n)^{-1} K^{jk}_{\alpha\beta}$, with
  scaling factor $\phi_n > 0$, and coefficient $K_{\alpha\beta}^{jk} \in \mathbb{R} $. 
 \begin{enumerate}
\item There exists a continuous function 
  $\mathcal{K}_{\alpha\beta}  \colon D \times D \to \RSet_{\geq 0}$ such
  that  
 \begin{align*}
\lim_{n\to\infty} R^n = 0 
\end{align*}
where
\begin{align*}
&R^n = n^{-1}\sup_{\alpha \in \mathbb{N}_q}\sup_{y\in [-1,1]^{qn} }\sum_{j \in \mathbb{N}_n , \alpha \in \mathbb{N}_q} &&\bigg(  n^{-1}\sum_{k\in \mathbb{N}_n , \beta \in \mathbb{N}_q}\bigg( \phi_n^{-1}K^{jk}_{\alpha\beta} \\ & &&  -   \mathcal{K}_{\alpha\beta}(x^j_n , x^k_n) \bigg) y^k_{\beta} \bigg)^2.
\end{align*}
\item There exists a constant $\mathfrak{c} > 0$ such that $|K^{jk}_{\alpha\beta}| \leq \mathfrak{c}$ always.
    \item The following limit is finite 
    \begin{equation}\label{eq: absolute summability of connectivity sparseness}
\lsup{n}(n\phi_n)^{-1} \sup_{\alpha \in \mathbb{N}_q}\sup_{j\in \mathbb{N}_n}  \sum_{k\in \mathbb{N}_n}\sum_{\beta\in\mathbb{N}_q} \chi\lbrace K^{jk}_{\alpha\beta} \neq 0 \rbrace < \infty.
    \end{equation}
    \item The limit $\lim_{n\to\infty}\phi_n$ exists (and could be zero).
    \item It holds that
    \[
\lim_{n\to\infty} n\phi_n = \infty .    
    \]
  \end{enumerate}
\end{hypothesis} 
Intuitively, $\mathcal{K}_{\alpha\beta}(x,y)$ is a function representing the average
connectivity from $y$ to $x$, and \cref{hyp:connections} is similar to an
assumption on the convergence of a Graphon  employed in many papers on
high-dimensional deterministic ODEs on sparse networks \cite{bramburger2023pattern}. 
There are several ways to employ these assumptions. One option is to simply
take the connectivity to be such that 
\rev{
$K^{jk}_{\alpha\beta} =
\mathcal{K}_{\alpha\beta}(x^j_n,x^k_n) \phi_n$,
}
hence $\{ K^{jk} \}_{jk} \subset \RSet^{q
\times q}$ are deterministic matrices indicating connection strengths; in
this case we say the particle model is \textit{with kernel matrices}, as it uses
directly the kernel functions $\mathcal{K}_{\alpha\beta}$.

Another option, which we refer to as \textit{models with random ternary matrices}, is to use
random  ternary matrices $\{ K^{jk} \}_{jk} \subset
\{ -1, 0, 1 \}^{q \times q }$, whose values represent inhibitory, absent, and
excitatory connections, respectively. In this case one assumes $K^{ab}_{\alpha\beta}$
is probabilistically dependent on $K^{ab}_{\eta\zeta}$ only if both $a = j$ and $b =
  k$; further, one chooses continuous functions
  $p^+_{\alpha\beta},p^-_{\alpha\beta} \colon D \times D \to
[0,1] \subset \RSet$ such that $\mathcal{K}_{\alpha\beta}(x,y) =
p^+_{\alpha\beta}(x,y) - p^-_{\alpha\beta}(x,y)$ and 
\begin{equation}\label{eq:probKernel}
  \begin{aligned}
    \mathbb{P}\big( K^{jk}_{\alpha\beta} & = -1 \big) = \phi_n p^-_{\alpha\beta}(x^j_n,x^k_n), \\
    \mathbb{P}\big( K^{jk}_{\alpha\beta} & = 1 \big) = \phi_n p^+_{\alpha\beta}(x^j_n,x^k_n), \\
    \mathbb{P}\big( K^{jk}_{\alpha\beta} & = 0 \big) = 1 - 
      \varphi_n p^+_{\alpha\beta}(x^j_n,x^k_n) - \varphi_n p^-_{\alpha\beta}(x^j_n,x^k_n).
  \end{aligned}
\end{equation}
In this case the kernel functions are linked to the probability measures for the
connectivity (possibly via rescaling $\mathcal{K}_{\alpha\beta}$ by the constant
$\mathfrak{c}$ in \cref{hyp:connections}). This method of constructing the random
graph (by choosing the probability of a connection to be determined by the spatial
positions of the nodes) is variously referred to as a stochastic block model, sparse
Erdos-Renyi random graph, or $W$-random graph \cite{borgs2019} in the literature.

\begin{remark}
  In the Supplementary Material (\cref{sec:supplementary}) we prove that models with random ternary matrices satisfy
  \cref{hyp:connections}. Our numerical examples use by default models with kernel
  matrices, albeit \cref{sec:numerics} contains numerical results for both model
  types, and comparisons between each of them and the mean field.
\end{remark}

\subsection{Nonlinearities, external deterministic forcing, and diffusion term} \label{Section Nonlinearities}
In order to complete the description of the particle system \cref{eq:particleModel},
we discuss the nonlinearities and forcing in the system.

Each neuron is described through a rate equation, with firing rate encoded in the
nonlinearity $f \colon \RSet^q \times \RSet^q$, which is defined componentwise by
\[
  (f(v))_{\alpha} := f_{\alpha}(v_{\alpha}), \qquad \alpha \in \NSet_q,
\]
where $\{ f_{\alpha} \}_\alpha$ are standard firing rate functions (more precise
hypotheses will be given later).

The network is subject to an external deterministic forcing $I \colon [0,T] \times D \to
\RSet^{nq}$, which appears in the particle model \cref{eq:particleModel} via the
position
 \[
   (I^j_t)_{\alpha} := \bigl(I(t,x^j_n)\bigr)_\alpha, 
   \qquad 
   (j,\alpha,t) \in \NSet_n \times \NSet_q \times [0,T].
\]

Similarly, the noise intensity is modelled via a deterministic mapping $G \colon [0,T]
\times D \to \RSet^{q \times q}$, from which we extract matrices $\{ G^j_t \}_{j,t}
\subset \RSet^{q \times q}$ acting on vectors $v \in \RSet^q$ as
\[
  (G^j_t v)_{\alpha} = \sum_{\beta \in \NSet_q} (G(t,x^j_n))_{\alpha\beta} v_\beta,
  \qquad 
  (j,\alpha,t) \in \NSet_n \times \NSet_q \times [0,T].
\]
\subsection{Initial Conditions}

The initial conditions are taken to be non-random constants.   Throughout almost all of this paper it is assumed that the initial conditions $\lbrace z^j \rbrace_{j \in \mathbb{N}_n}$ are such that the following condition is satisfied. 
\begin{hypothesis}[Initial Conditions] \label{hyp:spatialDistribution}
  There is a measure $\kappa \in \calP(D \times \mathbb{R}^{q})$ such that the empirical measure
  \[
    \hat \mu^n = \frac{1}{n} \sum_{j \in \NSet_n} \delta_{x^j_n , z^j} 
    \in \calP(D \times \mathbb{R}^{q})
  \]
  satisfies $\hat \mu^n \to \kappa \in  \calP(D \times \mathbb{R}^{q})$ as $n \to
  \infty $. For $x\in D$, write $\kappa_x \in \mathcal{P}(\mathbb{R}^q)$ to be
  $\kappa$ conditioned on its first variable $x$. It is assumed that $\kappa_x$ is
  Gaussian, with mean $m_0(x) \in \mathbb{R}^q$ and covariance matrix $V_0(x) \in
  \mathbb{R}^{q\times q}$. Furthermore it is assumed that $x \mapsto m_0(x)$ and $x
  \mapsto V_0(x)$ are continuous.
\end{hypothesis}
In Corollary \ref{corollary stopping time} we explain how this general assumption allows us to determine the limiting dynamics for a very broad class of initial conditions (in particular, if the initial conditions are sample from the large $T$ equilibrium distribution).


Combining all terms we obtain a complete, component-wise version of the particle
model \cref{eq:particleModel}, 
\begin{equation}\label{eq:particleComponent}
  \begin{aligned}
    & du^j_{\alpha,t} = \bigg[ \sum_{\beta \in \NSet_q}  
      \bigg(
    \begin{aligned}[t]
    & 
     - L_{\alpha\beta} u^j_{\beta,t}  
    + \frac{1}{n \phi_n} \sum_{k \in \NSet_n} K^{jk}_{\alpha\beta} f_{\beta}(
    u^k_{\beta, t}) 
      \bigg)
    +I^j_{\alpha,t} \bigg] dt  \\
    & + \sum_{\beta \in \NSet_q} G^j_{\alpha\beta,t} dW^j_{\beta,t},
    \quad \qquad (j,\alpha,t) \in \NSet_n \times \NSet_q \times \RSet_{ \geq 0},
      \end{aligned}
      \\
    & u^j_{\alpha,0} = z^j_\alpha, \qquad (j,\alpha) \in \NSet_n \times \NSet_q
  \end{aligned}
\end{equation}
 
We complete this section by making some assumptions on the functional inputs
presented above
\begin{hypothesis}[Functional inputs to the particle system]\label{hyp:functions}
  It holds that:
  \begin{enumerate}
    \item The average synaptic connectivity $\mathcal{K}_{\alpha \beta}$ is in $C(D \times D, \RSet )$ for all $\alpha,
      \beta \in \RSet^q$. 
    \item The firing rate function $f: \RSet^{q} \to \RSet^{q}$ is uniformly
      Lipschitz.
      \item The firing rate function $f$ is uniformly upperbounded.
    \item The external input function $I$ is in $C([0,T],C(D,\RSet^{q}))$.
    \item The noise intensity function $G$ is in $C([0,T],C(D,\RSet^{q \times q}))$.
    \item The noise intensity also satisfies
 
    \begin{equation}
        \inf_{t\in [0,T] , x\in D}\big| \det(G(t,x))\big|  > 0.
    \end{equation}
    \item The processes $\{W^j_t \colon t \in [0,T] \}_{j=1}^n$ are independent
      $\RSet^{q}$-valued Brownian Motions.
  \end{enumerate}
\end{hypothesis}
 
Henceforth we will refer to \crefrange{hyp:domain}{hyp:functions} as our
\textit{standing assumptions} and we assume they hold throughout the rest of the
paper.

Since the drift term in \eqref{eq:particleComponent} is Lipschitz, it is well-known
that there exists a unique strong solution
\cite{karatzasBrownianMotionStochastic1998}. By a `strong solution', it is meant that
for any valid probability space containing the Brownian Motions, initial conditions
and random connections, there exists precisely one stochastic process $u^j_t$
satisfying the above conditions.

\subsection{Main results}\label{ssec:mainResults} 
We wish to study the empirical measure associated to a solution of the system in the
interval $J = [0,T]$. For our treatment it is convenient to augment the orbit $\{ u^j_t
\colon t \in [0,T]\}$, so as to include also the positions $\{ x^j \}_j$, which have
trivial time-independent dynamics. We thus introduce the Banach space $(S_T, \| \blank
\|_{S_T})$ for
this extended orbit of the system, by setting 
\[
  S_T = D \times C([0,T],\RSet^q),
  \qquad 
  \| (x,u) \|_{S_T} = \| x \|_{\RSet^d} +  \| u \|_T,
\]
with $\| \blank \|_T$ given in \cref{eq:normT}, and introduce the empirical
measure 
\[
  \hat \mu^n_T = \frac{1}{n}\sum_{j \in \NSet_n}
  \delta_{s^j}
  \qquad 
  s^j = \bigl(x^j, \{ u^j_t \colon t \in [0,T] \}\bigr).
\]
The empirical measure lives in the space
\begin{equation}\label{eq:YTSpace}
  Y_T = \{ \mu \in \calP(S_T) \colon \mean^{s \sim \mu}\big[  \| s \|_{S_T} \big] < \infty \},
\end{equation}
and the topology of $Y_T$ is metrized by the Wasserstein Distance, defined as follows
\[
  d_{Y_T}(\mu,\nu) = \inf_{ \xi \in \Gamma(\mu,\nu)} 
  \mean^{(r,s) \sim \xi} \big[  \| r-s \|_{S_T} \big],
\]
where $\Gamma(\mu,\nu) \subset \calP(S_T \times S_T)$ is the set of all couplings of
$\mu$ and $\nu$, that is, the subset of $\calP(S_T \times S_T)$ so that
the law of the random variable $r$ under $\xi$ is identical to $\mu$, and
the law of the random variable $s$ under $\xi$ is identical to $\nu$.

Our first results concern the almost sure convergence of the empirical measure
\begin{theorem}\label{Theorem Almost Sure Convergence}
  With probability one the empirical measure $\hat \mu^n_T$ converges
  as $n \to \infty$ to a unique measure $\mu_T \in Y_T$.
\end{theorem}

We have specifically chosen a linear local interaction to ensure that $\mu_T$ admits
a very tractable dynamics (i.e. one does not need to solve a PDE), and our next result concerns $\bar{\mu}_t \in \mathcal{P}\big( D
\times \mathbb{R}^q \big)$, which we define to be the marginal of the limiting law at
time $t$. This marginal is characterised in terms of a Gaussian distribution, whose
mean and covariance satisfy a tractable integro-differential equation, resembling a
neural field system.

\begin{lemma} \label{Lemma Limiting Equations}
  The marginal $\bar{\mu}_t$ can be written as the following probability law: for
  measurable subsets $A \subset D$, and $B \subseteq
  \mathbb{R}^q$,
\begin{equation}\label{eq:marginal}
  \bar{\mu}_t(A \times B) = 
  \int_{A} \int_{B} \rho_q\big( m(t,x) , V(t,x) ,u \big)\, du \,d\ell(x), 
\end{equation}
where $\ell$ is as in \cref{hyp:spatialDistributioninitial}, and where $\rho_q$, $m$, and
$V$ are defined below. The function $\rho_q \colon \RSet^q \times \RSet^{q \times q}
\times \RSet^q \to \RSet_{>0}$ is the Gaussian density for a distribution with mean $m$
and invertible covariance matrix $V$, 
\[
  \rho_q(m,V,u) = 
  \frac{1}{\sqrt{(2\pi)^q \det(V)}}
  \exp\biggl( - \frac{1}{2} (u-m)^T\, V^{-1}\, (u-m) \biggr).
\]
The mappings $m \colon [0,T] \times D \to \RSet^q$, and $V \colon [0,T] \times D \to
\RSet^{q \times q}$ appearing in \cref{eq:marginal} are mean and covariance matrices
of Gaussian variables, in the sense that for any $(t,x) \in [0,T] \times D$ there
exists an $\RSet^q$-valued Gaussian variable $u_t(x)$ such that
\[
  \begin{aligned}
    & m(t,x) := \mean[ u_t(x) ]\\
    & V(t,x) := 
    \mean\bigl[ \big( u_t(x) - m(t,x) \big) \big(u_t(x) - m(t,x)\big)^T \bigr].
  \end{aligned} 
\]

Further, $t \mapsto ( m(t,\blank), V(t,\blank) )$ is the unique solution in
$C^1([0,T],C(D,\RSet^q \times \RSet^{q \times q}))$ to the following initial-value
problem 
\begin{equation}\label{eq:meanField}
  \begin{aligned}
  & \partial_t m(t,x) = - L m(t,x) 
    + \int_{D} \calK(x,y) F\bigl(m(t,y), V(t,y)\bigr) \ell(d y) + I(t,x),
    \\
  & \partial_t V(t,x) = - L V(t,x) - V(t,x) L^T + G(t,x)G^T(t,x),
  \\
  & m(0,x) = m_0(x)\\
  & V(0,x) = V_0(x) 
  \end{aligned}
\end{equation}
on $(t,x) \in [0,T] \times D$ in which $F \colon \RSet^q \times \RSet^{q \times q}
\to \RSet^q$ given by
\begin{equation}\label{eq:FDef}
F(m,v) = \int_{\mathbb{R}^q} \rho_q(m , v,x) f(x) dx,
\end{equation}
and with $L$, $I$, $G$, $\calK$, $f$, $m_0$, and $V_0$ given in Sections \ref{Section Connectivity} and \ref{Section Nonlinearities}. 
\end{lemma}

\begin{remark}
We remark that the mean field \cref{eq:meanField} is exact, in the sense that it does
not involve an approximation or a truncation, and does not require expressions for
the cross-variance at difference spatial positions to `close' the dynamics. For a
standard treatment of ODEs for the mean and variance of linear SDEs, we refer to \cite[Section
5.6]{karatzasBrownianMotionStochastic1998}.
\end{remark}

Most papers concerning the large size limiting behavior of neurons with sparse
disordered connections assume that the initial conditions are sampled independently
from the law of the random connections. A major strength of the results in this paper
is that we do not require this assumption. See also a similar result due to Coppini
\cite{coppini2022long}. We underscore this in the following corollary.
\begin{corollary} \label{corollary stopping time}
Suppose that all of the previous assumptions hold, except that now the initial conditions $\lbrace u^j_{\alpha,0} \rbrace_{j\in \mathbb{N}_n, \alpha\in \mathbb{N}_q}$ are arbitrary constants. Let $\tau_n$ be any valid stopping time such that with unit probability,
\begin{align}
\lim_{n\to\infty} n^{-1}\sum_{j\in \mathbb{N}_n} \delta_{x^j_n , u^j_{\tau_n}} = \kappa.
\end{align}
Then if we define $z^j  := u^j_{\tau_n}$, the result in Theorem \ref{Theorem Almost Sure Convergence} holds (with the time defined relative to the stopping time, i.e. $t \to t-\tau_n  $).
\end{corollary}
To see why this corollary must be true, one only needs to verify that Hypothesis \ref{hyp:spatialDistribution} is indeed satisfied. 

We refer the reader to \cite{karatzasBrownianMotionStochastic1998} for a definition of stopping times. Essentially, it is any random time that `cannot know the future': for example the following `first-hitting-time' is a valid stopping time, for some $\zeta \in \mathcal{P}\big( D \times \mathbb{R}^q\big)$ and $\epsilon_n > 0$,
\[
\tau_n = \inf \big\lbrace t\geq 0 \; : \; d_W\big( n^{-1}\sum_{j\in \mathbb{N}_n} \delta_{x^j_n, u^j_t} , \zeta \big)  \leq \epsilon_n  \big\rbrace .
\]
By contrast the following is not a valid stopping time, for any measurable $A \subseteq \mathbb{R}^{nq}$,
\[
\tilde{\tau}_n = \inf\big\lbrace t \geq 0 \; : \; u_{t+1} \in A \big\rbrace
\]
\subsection{Large Deviations Results}

The importance of understanding how noise can induce rare events in biological systems is increasingly recognized \cite{keener2011perturbation,bressloff2014path}. A key early study on this phenomenon was conducted by Newby, Keener and Bressloff \cite{newby2013breakdown,bressloff2014path,newby2014spontaneous}. They determined that it is possible that small amounts of noise can induce rare events in excitable systems (such as spontaneous production of an action potential \cite{keener2011perturbation,newby2014spontaneous}).

It is widely conjectured that noise-induced transitions between attractor states could be essential to the brain's proper functioning \cite{deco2017dynamics,breakspear2017dynamic}. Such transitions include UP / DOWN transitions
\cite{shu2003turning}, the wandering of bumps of
activity in the visual cortex \cite{maclaurin2020wandering}, and stochastic models of binocular rivalry \cite{moreno2007noise}. For this reason, we also prove a Large Deviations Principle: this gives an asymptotic characterization of the probability of a transition path. It is significant to notice that the Large Deviations rate function is the same as the rate function with averaged connections (defined in the course of the proof), In other words, even for rare events, rare fluctuations in the sparse structure of the graph have a negligible effect.
\begin{theorem}\label{Theorem LDP Coupled System}
There exists a lower-semicontinuous function $\mathcal{J}_T: Y_T \to
\mathbb{R}^+$ (specified further below in  \cref{eq: transformed rate function})
such that for any sets $\mathcal{A},\mathcal{O} \subset Y_T$, with
$\mathcal{A}$ closed and $\mathcal{O}$ open, 
\begin{align}
\lsup{n} n^{-1}\log \mathbb{P}\big( \hat{\mu}^n_T  \in \mathcal{A} \big)  &\leq - \inf_{\mu \in \mathcal{A}} \mathcal{J}_T(\mu) \\
\linf{n} n^{-1}\log \mathbb{P}\big( \hat{\mu}^n_T  \in \mathcal{O} \big)  &\geq - \inf_{\mu \in \mathcal{O}} \mathcal{J}_T(\mu) .
\end{align}
Also $\mathcal{J}_T$ has compact level sets.
\end{theorem}

\section{Noise-induced Turing-like bifurcation}\label{sec:Turing}
We illustrate the noise-induced mechanism for pattern-formation in its simplest
form, by considering a one population network ($q=1$) on a ring of width $2l$, $D =
\RSet/2l\ZSet$ with distance dependent kernel $\calK(x,y) = A(x-y)$, where $A \in
C_p(2l,\RSet)$, the space of continuous real-valued $2l$-periodic functions. 
We consider a network with $L =1$, $I(t,x) \equiv 0$, and
time-independent, homogeneous noise intensity $G(t,x) \equiv \sigma$, for some
$\delta \geq 0$. 
To illustrate the Turing bifurcation numerically we will select $l=20\pi$, and we
  highlight that, upon rescaling space and parameters in the kernel $A$, the model can be
equivalently posed on a ring of width $\pi$ or $2\pi$, thereby modelling orientation
or direction preference, respectively \cite{ermentroutMathematicalFoundationsNeuroscience2010,
coombesNeurodynamicsAppliedMathematics2023}.

We show that a homogeneous steady state $(m_*,v_*)$ of \cref{eq:meanField} which is
linearly stable in the noiseless network ($\sigma =0$) becomes linearly unstable
to spatially-periodic perturbations for sufficiently large noise intensity $\sigma$.

Firstly, we note that the mean-field evolution for the covariance matrix $V$ in
\cref{eq:meanField} is decoupled from the mean evolution, and its dynamics is
uniformly contracting over $D$, as the following lemma states.
\begin{lemma}\label{lem:contraction}
Suppose that $G_t(x)$ is independent of $t$ and that there is a positive constant
$\gamma$ such that $\max \{ \real \lambda \colon \lambda \in \sigma(L)\} < -\gamma$,
where $\sigma(L)$ is the spectrum of the matrix $L$. Then there exists a unique $V_*(x) \in
C(D, \mathbb{R}^{Q\times Q})$ and constants $\beta, \epsi >0$ such that for all $V_0
\in C(D,\RSet^{q \times q})$ and all $x \in D $,
  \[
  \| V(t,x) - V_*(x) \|_{C(D,\RSet^{q  \times q})} 
  \leq \beta e^{-\epsi t}
  \| V_0 - V_*(x) \|_{C(D,\RSet^{q  \times q})},  
  \qquad 
  t \in \RSet_{\geq 0}.
\]
\end{lemma}
\begin{proof} See \cref{sec:appendixProofContraction}.
\end{proof}
 
For the network under consideration, homogeneous steady states are identified with
$(m_*,v_*) \in \RSet^2$ satisfying the system
\begin{equation}\label{eq:HomSteadyState}
  0 = -m_* + F(m_*,\sigma^2/2) \int_{-l}^{l} A(x) \,d x, \qquad v_* = \sigma^2/2,
\end{equation}
from which we deduce that homogeneous steady states are parametrised by $\sigma$.
Assuming a homogeneous steady state $(m_*,v_*)$ exists, to assess linear its stability we must study
the asymptotic behaviour of solutions to \cref{eq:meanField} of the form
$\bigl(\tilde m(t,x), \tilde
v(t,x)\bigr) = \bigl(m_* + m(t,x), v_* + v(t,x) \bigr)$ for small $(m,v)$.

From \cref{lem:contraction}, however, we know that $v(t,x) \equiv v_* = \sigma^2/2$
is a stable equilibrium of the covariance equation, which is decoupled from the
dynamics of the mean $m$, hence perturbations to the initial conditions around $v_*$
decay exponentially fast. Therefore we focus on the long-term behaviour of
solutions to \cref{eq:meanField} of the form $\bigl(\tilde m(t,x), \tilde
v(t,x)\bigr) = \bigl(m_* + m(t,x), v_* \bigr)$ for small $m$: in this setup the
noise-induced Turing-like bifurcation is a Turing-like bifurcation 
of homogeneous steady states of \cref{eq:meanField}, in the parameter $\sigma$ for
the problem
\begin{equation}\label{eq:mEquation}
  \partial_t m(t,x) = - m(t,x) + \int_{-l}^l A(x-y) F\bigl(m(t,y), \sigma^2/2\bigr)d y.
\end{equation}

We will proceed formally. For a rigorous centre-manifold reduction of
neural field equations around several instabilities, including a Turing-like
bifurcation, we refer to \cite{avitabileLocalTheorySpatioTemporal2020}. We also refer to
\cite{carrillo2023noise} for a rigorous treatment of Turing-like bifurcations in
a nonlinear, nonlocal Fokker-Planck equation for neuronal dynamics.

Small perturbations $m(t,x)$ to $m_*$, under $v(x,t) \equiv \sigma^2/2$ are governed
by the linearised evolution equation
\begin{equation}\label{eq:linear}
    \partial_t  m(t,x) = - m(t,x) +  D_mF(m_*,\sigma^2/2) \int_{-l}^{l} A(x-y)
    m(t,y)\,d y, 
\end{equation}
for $(t,x) \in \RSet_{ \geq 0} \times \RSet\ 2l\ZSet$, in which we have denoted by
$D_mF$ the derivative of $F$ with respect to $m$, which is guaranteed to exist by
\cref{eq:FDef}.
Linear stability is determined by long-term behaviour of solutions to the
problem above, seen as an ODE on the Banach space $C_p(2l,\CSet)$. Expressing the
kernel $A \in C_p(2l,\RSet)$ as a Fourier series
\[
  A(x) = \sum_{k \in \ZSet} A_k \phi_k(x), \qquad A_k = \frac{1}{2l}\int_{-l}^{l}
  A(x) \overline{\phi}_k(x) \,d x, \qquad \phi_k(x) = e^{i k\pi x/l}, \qquad k \in
  \ZSet,
\]
and using the periodic convolution theorem, we find that
\cref{eq:linear} admits solutions
\[
  m_k(t,x) = \exp (\lambda_k t + i k \pi x/l),  
  \qquad 
  \lambda_k = -1 + D_m F(m_*, \sigma^2/2) 2l A_k
  \in \CSet,
  \qquad k \in \ZSet.
\]

A noise-induced Turing-like bifurcation occurs along on a branch of homogeneous
steady states $\{ m_*(\sigma) \colon \sigma \in I \subset \RSet\}$ if one of the
eigenvalues of $\lambda_k$ crosses the imaginary axis with nonzero speed, as
$\sigma$ varies. More precisely, the bifurcation occurs at $\sigma=\sigma_c$ if there
exist $(\sigma_c,k_c) \in I \times \ZSet$ such that the mapping
\begin{equation}\label{eq:gammaCurve}
  \gamma_{k} \colon I \to \RSet, \qquad  \sigma \mapsto \real \big[ -1 + D_m F(m_*(\sigma), \sigma^2/2) 2l A_k \big]
\end{equation}
satisfies the conditions $\gamma_{k_c}(\sigma_c) = 0$ and $\gamma'_{k_c}(\sigma_c)
\neq 0$.

For suitable choices of kernel and firing rate function it is possible that a homogeneous
equilibrium in the absence of noise, $m_*(0)$, is linearly stable while, upon
increasing $\sigma$, the equilibrium $m_*(\sigma_c)$ becomes linearly unstable to
perturbations with wavelength $k_c \pi/l$. 

\section{Numerical experiments}\label{sec:numerics}
We now present numerical simulations on the particle and mean-field system, and we
refer to \cite{avitabileMacLaurinCodes} for a public repository with our codes. 
The particle system is simulated using a standard Euler-Maruyama scheme with timestep $dt
= 0.01$. We discretise the mean field with the spectral collocation method proposed
in \cite{rankinContinuationLocalisedCoherent2014} which is spectrally
convergent~\cite{avitabileProjectionMethodsNeural2023}, and use Matlab's  in-built
\textsc{ODE45} for time stepping. We employed numerical bifurcation analysis tools
developed in \cite{rankinContinuationLocalisedCoherent2014,avitabileTutorial}. 


\subsection{Example of noise-induced Turing-like bifurcation} \label{ssec:turingExample} 

 \begin{figure}
   \centering
   \includegraphics{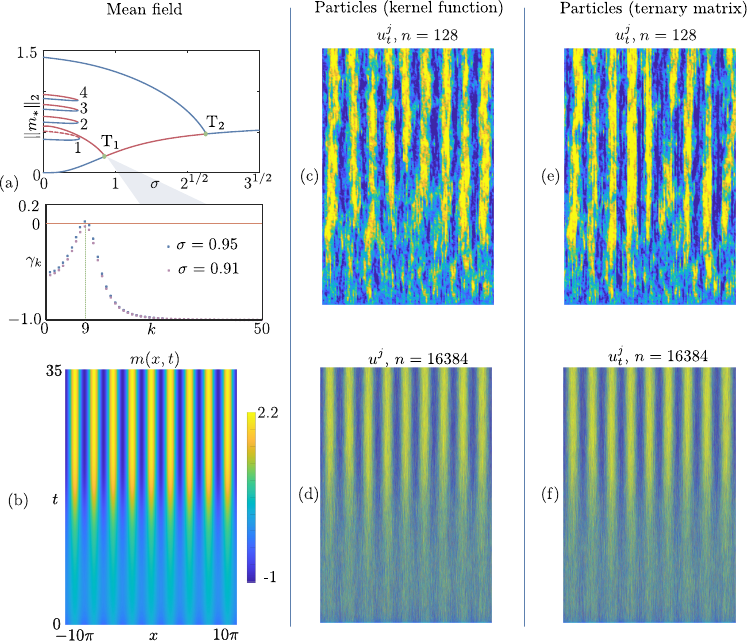}
   \caption{
     Noise-induced Turing-like bifurcation for the particle and mean-field
     model with one population ($q=1$), posed on a ring of width $2l$, with synaptic
     kernel \cref{eq:LTKernel}, neuronal firing rate \cref{eq:particleFiringRate}, and
     mean-field firing rate \cref{eq:firingRateMeanField}. (a) Bifurcation diagram of
     steaady states $m_*(x)$ to the mean-field system in the parameter $\sigma$. The
     homogeneous steady state undergoes Turing-like bifurcations at $T_1$ and $T_2$
     that generate one unstable (red) and one stable (blue) branch of
     spatially-periodic steady states, respectively. The mean field system also
     supports steady states that are spatially localised, with multiple bumps
     (branches with 1--4 bumps are reported here). In the inset, the values $\{ \gamma_k
     \}_k$ (defined as in \cref{eq:gammaCurve}), confirm that a Turing-like bifurcation
     of the homogeneous steady state is located between $\sigma = 0.95$ and $\sigma =
     0.91$, with critical wavenumber $k_c = 9$. (b) Numerical simulation of the
     mean-field system for $\sigma = 1 > \sigma_c$ showing that the system reaches a
     stable, spatially-periodic steady state which, using the bifurcation diagram in
     (a), is deduced to be the one emanating from $T_2$.
     (c,d) The simulation in (b) is repeated for the particle system
     \cref{eq:particleModel} with kernel function for $n = 128$ and $n = 16384$,
     respectively. (e,f) As
     in (c,d) but for a particle model with one sample of ternary matrix specified
     using \cref{eq:probKernel,eq:pPlusMinus}, \rev{$\phi_n \equiv 1$}, and {\color{magenta} driven by the same Brownian increments} as in (c,d).
     Parameters: $L =1$, $l =10 \pi$, $I(t,x) \equiv 0$, $G(x,t) \equiv \sigma$, $B =
     0.4$, $C = 1$, $\alpha = 10$, $\theta = 0.9$, $m_0(x) = 0.3\cos(k_c \pi x/l)$. 
   }
 \label{fig:turingAndComparison}
 \end{figure}

To exemplify the phenomenon, we consider a network with $q=1$ on a ring of width $2l$, $D =
\RSet/2l\ZSet$ with distance dependent kernel \rev{$\calK(x,y) = 2l A(x-y)$}, where 
\begin{equation}\label{eq:LTKernel}
  A(x) = C e^{-B|x|}(B\sin |x| + \cos x ), \qquad B, C \in \RSet_{0}
\end{equation}
linear coupling $L =1$, forcing $I(t,x) \equiv 0$, $G(t,x) \equiv \sigma$, and with
firing rate function 
\begin{equation}\label{eq:particleFiringRate}
  f(u) = \Phi(\alpha(u-\theta)), \qquad \Phi(u) = \frac{1}{2}\biggl[1 + \erf\biggl(\frac{u}{\sqrt{2}}\biggr)\biggr],
  \qquad 
  \alpha \in \RSet_{>0},
  \quad
  \theta \in \RSet,
\end{equation}  
which results in a mean-field firing rate of the form
\cite{touboulNoiseInducedBehaviorsNeural2012a}
\begin{equation}\label{eq:firingRateMeanField}
  F(m,v) = \Phi\biggl(\alpha \frac{m-\theta}{\sqrt{1+\alpha^2 v}}\biggr).
\end{equation}
The kernel $\cal{K}$ is known to support localised
stationary solutions arranged in a snaking bifurcation structure for neural fields
\cite{laingMultipleBumpsNeuronal2002,rankinContinuationLocalisedCoherent2014}, in
additions to spatially-periodic solutions emerging from a Turing bifurcation. 

We employed numerical bifurcation analysis tools to study the
bifurcation structure of steady states to of the mean-field equation in the parameter
$\sigma$ (see \cref{fig:turingAndComparison}(a)). The primary bifurcation
$T_1$ is subcritical, and a branch of unstable spatially-periodic steady states
emerges from a branch of homogeneous steady states (bottom branch). In contrast, the
secondary bifurcation, $T_2$, is supercritical and gives rise to stable
spatially-periodic equilibria. Further, the model supports localised solutions
with 1--4 core bumps, which are not arranged in a snaking diagram here, but rather on
disconnected branches. In the inset of \cref{fig:turingAndComparison}(a) we show $\{
\gamma_k \}_{k \in \NSet_{50}}$ for two values of $\sigma$ near $T_1$, which provides
evidence of a bifurcation at $\sigma_c \in (0.91,0.95)$ with wavenumber $k_c = 9$. 

\subsection{Kernel and ternary-matrix models for varying number of particles}
We carried out time simulations of the particle system with varying numbers of
neurons, and of the mean-field equation to confirm the predictions of the numerical
bifurcation analysis in \cref{fig:turingAndComparison}(c--f). Our first experiment,
in \cref{fig:turingAndComparison}(c,d) confirms the predictions by keeping the same
parameters as in \cref{fig:turingAndComparison}(b), but implements a particle model
with various $n$, and kernel matrix induced by \rev{$\calK(x,y) = 2l A(x-y)$} (see the end of
\cref{Section Connectivity} to recall this setup), \rev{and $\phi_n \equiv 1$}. 

Further, we repeat the experiment in \cref{fig:turingAndComparison}(c,d) by retaining
parameters and Brownian increments, but for one sample of a model with ternary-based
matrix. The kernel satisfies $| \mathcal{K}(x,y) | \leq  1$, hence we assign probabilities as
in \cref{eq:probKernel} with 
\begin{equation}\label{eq:pPlusMinus}
  p^+(x,y) = \max\bigl( \mathcal{K}(x,y), 0\bigr), \qquad p^-(x,y) =
  \max\bigl(-\mathcal{K}(x,y), 0\bigr), 
  \qquad \rev{\phi_n \equiv 1}.
\end{equation}
Further numerical evidence that the kernel- and ternary-matrix models agree with each
other and with the mean-field model is given
in \cref{fig:convergence-new}(b), in which we superimpose $m(x_j,T)$ (mean field) and
$u^j_T$ (particle models) for simulations leading to a localised and a
spatially-periodic stable state. The kernel- and ternary-matrix models do not display
appreciable differences hence, while the figures contain exactly three profiles at every point
$x_j$, we display in the foreground the kernel-matrix model profile for $x \leq  0$, and the
ternary-matrix one for $x >0$, to highlight the similarity.

The latter numerical experiments bring naturally the question of how the particle
system converges to the mean-field limit as the number $n$ grows, and we now turn our
attention to this.

\begin{figure}
  \centering
  \includegraphics{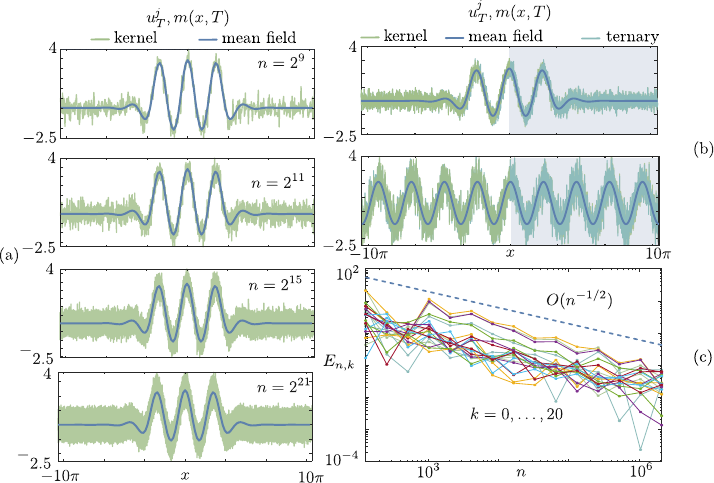}
  \caption{Convergence test for the particle and mean-field
    models with one population ($q=1$), posed on a ring of width $2l$, with synaptic
    kernel \cref{eq:LTKernel}, neuronal firing rate \cref{eq:particleFiringRate}, and
    mean-field firing rate \cref{eq:firingRateMeanField}. (a): Time simulations up to
    time $T=35$ for increasing values of $n$ approach a localised stationary state
    with 3 bumps in the core; we see that the profiles of the particle simulations
    ($u^j_T$, green) do not converge pointwise to the mean-field solution ($m(x,T)$,
    blue). 
      (b) We superimpose the solution at $t=T$ of the mean field and both kernel- and
      ternary-matrix particle models with $n = 2^{12}$ for the simulation in (a), and
      for the one in \cref{fig:turingAndComparison}(c--f). The kernel- and
      ternary-matrix models do not display appreciable differences; note that the
      figures in (b) contain exactly three profiles at every point $x_j$, but we display
      in the foreground the kernel-matrix model profile for $x \leq  0$, and the
      ternary-matrix one for $x >0$, to highlight the similarity.
    (c) The numerical
    simulations show that solutions to the particle system in (a) converge in weak
    sense to the mean-field state. The convergence is measured by $E_{n,k}$ as
    defined in \cref{eq:error}, which occurs as an
    $O(n^{-1/2})$ for $k = 0,\ldots,20$. Parameters $L =1$, $l =10 \pi$, $I(t,x)
    \equiv 0$, $G(x,t) \equiv \sigma $, $B = 0.4$, $C = 1$, $\alpha = 10$, $\theta =
    0.9$;
      in (a), (b, top panel), and (c) $\sigma = 0.45$, and $m_0(x) = 5/\cosh(0.25 x)$;
      in (b, bottom panel) $\sigma = 1$, $m_0(x) = 0.3\cos(9 \pi x/l)$; \rev{in the
      particle models we set $\phi_n \equiv 1$.}
}
  \label{fig:convergence-new}
\end{figure}

\subsection{Considerations on convergence}\label{sec:considerationsConvergence}

The empirical measure $\hat \mu^n_T$ is a mathematical object 
that facilitates an understanding of how the $n$-dimensional particle system can
converge to a continuum limit \cite{sznitman1991topics}. Following from \cref{Theorem
Almost Sure Convergence}, we know that $\hat \mu^n_T$ converges to a unique measure
$\mu_T$ weakly in $Y_T$, with probability one.

We are interested in characterizing the convergence of the marginals at particular
times, hence we do not address, numerically, the convergence of the probability
distribution across multiple times. We thus fix a particular time $t$ for which we test convergence. 

Let us suppose that $L^2(D)$ possesses a basis of orthonormal functions $\{ \phi_k
\}_{k \in \NSet} \subset C(D)$. A natural choice would be the
eigenvectors of the diffusion operator. Such a set is convergence-determining,
meaning that two measures converge weakly if and only if their expectations of
functions in this set converge. We thus test functions of the
form $f_{k,\alpha} , g_{k,\alpha,\beta}: D \times \mathbb{R}^q \to \mathbb{R}$
\[
  f_{k,\alpha}(x,u) = \phi_k(x) u_{\alpha}, \qquad 
  g_{k,\alpha,\beta}(x,u) = \phi_k(x) u_{\alpha} u_{\beta}, \qquad (k,\alpha) \in
  \NSet \times \NSet_q  
\]
for the mean and variance of $\hat \mu^n_t$, respectively. Note that technically,
these functions are not bounded (in contrast to what is required for 
two measures to converge weakly). However for our particular problem their expectations will also
converge. This is proved in the Supplementary Materials (\cref{sec:supplementary}). 

One can thus study (numerically) the convergence of the following errors
\[
  \begin{aligned}
  & E^m_{n,k}(t) = \max_{ 1\leq \alpha \leq q} 
  \Bigg|
  \frac{1}{n} \sum_{j \in \NSet_n} \phi_k(x^j_n)u^j_{\alpha,t}
  -\frac{1}{|D|}
  \int_{D} \phi_k(x) m_{\alpha}(x,t) \,d x
  \Bigg|, \\
  & E^V_{n,k}(t) = \max_{ 1\leq \alpha, \beta \leq q  } 
  \Bigg|
  \frac{1}{n} \sum_{j \in \NSet_n} \phi_k(x^j_n)u^j_{\alpha,t}u^j_{\beta,t}
  - \frac{1}{|D|}
  \int_{D} \phi_k(x) V_{\alpha\beta}(x,t) \,d x
  \Bigg|, \\
  \end{aligned}
\]
for which we expect that asymptotics are of the form
\[
  E^m_{n,k}(t), \, E^V_{n,k}(t) \in O(n^{-1/2}), \qquad \text{as $n \to \infty $},\qquad (k,t) \in \NSet \times [0,T]
\]
This rate of convergence holds in the original work of Sznitman \cite{sznitman1991topics},
and one also expects a Central Limit Theorem to hold (although this has not been
proved in this paper). The disordered connections complicate an easy adaptation of
the CLT in  \cite{sznitman1991topics}. Coppini, Lucon and Poquet have proved a
Central Limit Theorem for a similar disordered model \cite{coppini2023central}.

\subsection{Convergence}\label{ssec:convergence} 
To test convergence in \cref{fig:convergence-new} we have run a time simulation of the
model with the new kernel up to $T = 35$, which was sufficient for the solution to
approach a localised steady state with $3$ bumps at the core. In addition to the
mean-field simulation (in blue in \cref{fig:convergence-new}(a)), we have run simulations
of the particle system (in green in \cref{fig:convergence-new}(a) for various number of
$n$ between $2^8$ and $2^{21}$. The data in \cref{fig:convergence-new}(a) shows that the
particle system does not converge pointwise as $n \to \infty$, as we see an
increasing variability in the spatial profiles $u^j_T$ as $n \to \infty$. This is
entirely what one would expect from our main result in Theorem \ref{Theorem Almost
Sure Convergence}, because the limiting Gaussian measure $\mu_T$ has nonzero variance.

The results in Theorem \ref{Theorem Almost Sure Convergence} (and following discussion)
however, only concern the weak convergence of the empirical measure
(see \cite{billingsley2013convergence} for more explanation of this topology), hence we consider the following error
\begin{equation}\label{eq:error}
  E_{n,k} = \bigg| \frac{2l}{n} \sum_{j=1}^{n}  \phi_k(x^j) u^j_T - 
  \int_{-l}^{l} \phi_k(x) m(x,T) \,d x
  \bigg|,
  \qquad \phi_k(x) = \exp(i k \pi x/l),
\end{equation}
where $| \blank |$ denotes the absolute value for complex numbers. In
\cref{fig:convergence-new}(c) we plot $E_{n,k}$ as a function of $n$ for $k=0,\ldots,20$,
and we find evidence that $E_{n,k} = O(n^{-1/2})$ for the tested values of $k$. This
is precisely the rate of convergence that one expects from the Central Limit Theorem
\cite{graham1997stochastic,coppini2023central} (although we have not proved a CLT in
this paper).

\subsection{Spiral wave simulations} \label{ssec:spiral}
\begin{figure}
  \centering
  \includegraphics{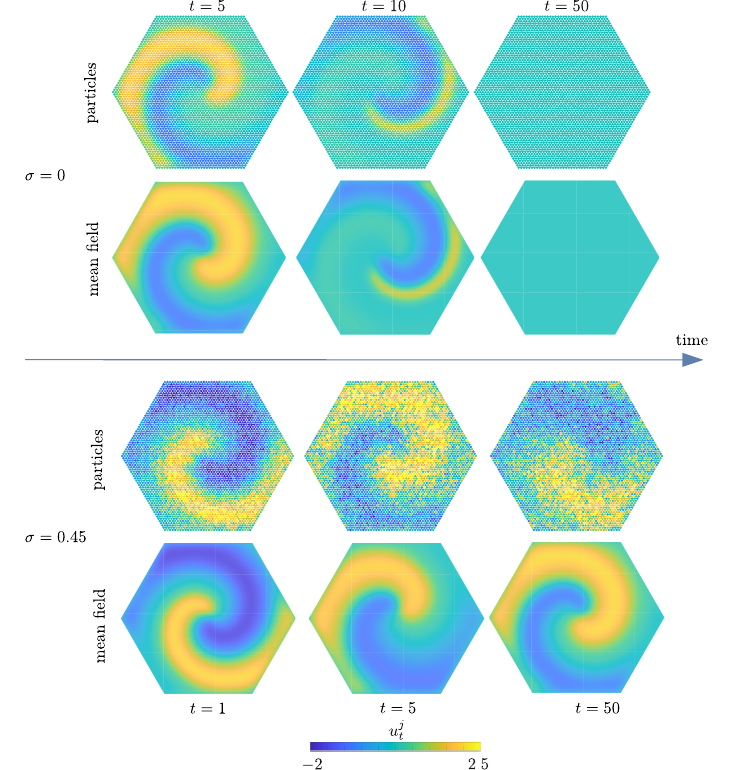}
  \caption{
    Comparison between mean-field and particle model for the spiral wave
    simulations of \cref{fig:patterns}(b) for $\sigma = 0$ (top) and $\sigma =0.45$
    (bottom). In both cases a coherent structure is sustained only for sufficiently
    large values of $\sigma$. The wave in the particle simulation is influenced by
    noise, as expected. In the
  \href{https://figshare.com/articles/media/Animations_for_the_paper_Neural_fields_and_Noise-Induced_Patterns_in_Neurons_on_Large_Disordered_Networks_/26763037?file=51540848}{accompanying
  animation}, where the simulation is carried out for $t \in [0,100]$ we see that the
  spiral wave in the particle model departs in phase from the one in the mean-field
(see figure for $t = 50$) but regains phase on longer times. We highlight 
that, as discussed in \cref{ssec:convergence}, we only expect weak (not pointwise)
convergence. \rev{In the particle model we set $\varphi_n \equiv 1$.}}
  \label{fig:sprialComparison}
\end{figure}
Spiral waves experiments of
\cref{fig:patterns}(b) have been carried out on a hexagonal cortex $D$ inscribed in
a disk of radius $30$, in which neurons are uniformly distributed on a triangular grid
(for details about meshing, and a similar spiral-wave simulation see \cite[Tutorial
2]{avitabile2024}). See also recent work by Lucon and Poquet that identified oscillatory behavior in coupled Fitzhugh-Nagumo neurons \cite{luccon2021periodicity}.

The synaptic connection for the simulation are sparse, determined by a
distance-dependent kernel $\calK(r,r') = \rev{| D|} \nu A(\| r - r' \|)$, 
\rev{where $|D|$ is the Lebesgue measure of $D$,}
with $\nu \in
\RSet_{>0}$ and
\[
  A(x)=
  \begin{cases}
    a(x) & \text{if $|a(x)|> 10^{-3}$,} \\
    0           & \text{otherwise,}
  \end{cases}
  \qquad 
  a(x) = \int_{0}^{\infty }J_{0}(xs)\frac{s}{s^4+s^2+1} \,ds.
\]
The dynamics of the model includes a recovery variable therefore we have $q = 2$ in
the particle system, and we set
\[
  L = 
  \begin{bmatrix}  
    L_{11} & L_{12} \\
    L_{21} & L_{22}
  \end{bmatrix}, 
  \qquad 
  f = 
  \begin{bmatrix}  
    \nu \Phi(\alpha(u-\theta)) \\
    0
  \end{bmatrix},
\]
with $\nu \in \RSet_{\geq 0}$ and $\Phi$ as in \cref{eq:particleFiringRate}. In the
simulation we also set $I(t,x) \equiv 0$ and $G(t,x) \equiv \sigma$.
It is possible to carry out a numerical bifurcation analysis for the spiral wave
pattern in the parameter $\sigma$ for the mean-field equations (see for instance
\cite{laingSpiralWavesNonlocal2005}), along the lines of what was done for the ring
model. We did not take this route here, and inspected the spiral wave onset with
direct numerical simulation. The parameters in \cref{fig:patterns}(b) have been
chosen so that the \textit{mean-field model} with assigned initial condition $m(0,x)$
does not exhibit a stable spiral wave for $\sigma = 0$; in a set of simulations where
$m(0,x)$ but $\sigma$ was increased, we found values of $\sigma$ for which a stable spiral
wave is supported by the mean-field model. We then carried out numerical simulations
for the particle model, to find an analogous behaviour. \Cref{fig:patterns} shows the
spiral wave onset in the particle model, whereas \cref{fig:sprialComparison}
compares simulations in the mean field and particle model. From the figure and the
\href{https://figshare.com/articles/media/Animations_for_the_paper_Neural_fields_and_Noise-Induced_Patterns_in_Neurons_on_Large_Disordered_Networks_/26763037?file=51540848}{accompanying
animation} we see that the mean field and particle simulations produce both spiral
waves, albeit the structure in the particle model displays transient phase slips, and
it is visibly altered by noise. This is to be expected, in line of the weak
convergence mode discussed in \ref{sec:considerationsConvergence}. Incidentally the
transient phase slips are also entirely expected; this phenomenon is well-known in
the Kuramoto model \cite{bertini2014synchronization}, see also
\cite{maclaurin2023phase,adams2025isochronal}.



\section{Proofs}\label{sec:proof}

The proofs employ the theory of Large Deviations \cite{Dembo1998}, which
characterizes the exponential asymptotics of the probability distribution for the
system. The proofs proceed by comparing the disordered system to progressively
simpler systems. In Section \ref{Section Large Deviations Uncoupled}, we note the
Large Deviations for the empirical measure generated by the Brownian Motions (this is
a corollary of Sanov's theorem). In Section \ref{subsection approximate heterogeneous
averaged}, we prove a Large Deviation Principle for the system with averaged
interactions by showing that its associated empirical measure can be written as a
continuous mapping of the uncoupled empirical measure. This method of proof appears
to have been pioneered by Tanaka \cite{tanaka1984limit}, and more recently has been
extended to spatially-distributed systems in \cite{maclaurin2024large}. Next, in
Section \ref{Section Disoredered Coupling}, we prove a Large Deviation Principle for
the system with disordered interactions by showing that the Radon-Nikodym derivative
is relatively small, with extremely high probability. Finally, in Section
\ref{Section Limiting Equations}, we prove that the Large Deviations rate function
has a unique zero, and this yields the limiting equations. Extra details are provided
in the Supplementary Materials (\cref{sec:supplementary}).


\subsection{Large Deviations of the Uncoupled System}
\label{Section Large Deviations Uncoupled} In this section we outline the Large
Deviations of the empirical measure generated by independent Brownian Motions. This
theory is already known and is essentially due to Sanov \cite{Dembo1998}. The result
in this section is useful because in the following section we demonstrate that the
empirical measure for the coupled system can be obtained by applying a continuous
transformation to the empirical measure resulting from the Brownian Motions.

We define the $\RSet^q$-valued random variables $\tilde W^j_{t}$ with components
given by
\[
  \tilde{W}^j_{\alpha,t} := (\tilde W^j_t)_{\alpha} = \int_0^t \sum_{\beta = 1}^qG_{\alpha\beta}(s,x^j_n) dW^j_{\beta,s}.
\]
We introduce the Banach space $(B_T,\| \blank \|_{B_T})$ with
\[
  B_T = D \times \mathbb{R}^q \times C\big( [0,T], \mathbb{R}^q \big), 
  \qquad 
  \| (x,z,w) \|_{B_T} = \| x \|_{\RSet^d} + \|  z \|_{\RSet^q} + \| w \|_{T},
\]
and study the convergence of the following empirical measure up to time $T$
\begin{equation}\label{eq:muTildeDef}
  \tilde{\mu}^n_T = n^{-1} \sum_{j \in \NSet_n} \delta_{b^j}
  \qquad 
  b^j = \bigl(x^j,u_0^j,\{ \tilde W^j_{t} \colon t \in [0,T] \}\bigr),
\end{equation}
which lives in the space
\begin{equation}\label{eq:XTSpace}
  X_T = \{ \mu \in  \calP(B_T) \colon \mathbb{E}^{b \sim
  \mu}[ \| b \|_{B_T} ] < \infty \}.
\end{equation}
The topology of $X_T$ is metrized by the
Wasserstein Distance, defined as follows. For $\mu,\nu \in X_T$, we define
\[
d_{X,T}(\mu , \nu) = \inf_{\xi \in \Gamma(\mu,\nu)} \mathbb{E}^{(b,\tilde b) \sim \xi}\big[ \| b - \tilde b \|_{B_T} \big],
\]
where as previously $\Gamma(\mu,\nu)$ is the set of all couplings between $\mu$ and $\nu$.


Next we define the rate function $\mathcal{I}_T(\mu)$ that governs the Large Deviations of
$\tilde{\mu}^n_T$. To this end, for $x\in D$, let $\mathcal{W}_T(x) \in
\mathcal{P}\big( \mathbb{R}^q \times C([0,T],\mathbb{R}^q) \big)$ be the law of
Gaussian processes $\{ w_{\alpha ,t}(x)) \colon
\alpha \in \NSet_q, \; t \in [0,T]\}$ with the following covariance structure: 
\begin{enumerate}
  \item  For any $(\alpha,t) \in \NSet_q \times [0,T]$, $w_{\alpha ,t}(x)$ has zero mean;
  \item For any $\alpha, \beta \in \NSet_q$ it holds
    \[
      \begin{aligned}
      & \mathbb{E}\big[ w_{\alpha,s}(x) w_{\beta,t}(x) \big]  
        = \int_0^s \sum_{\gamma=1}^q G_{\alpha\gamma,r}(x) G_{\beta\gamma,r}(x) dr,
         \qquad s\leq t \leq T.
      \end{aligned}
    \]
\end{enumerate}
We first stipulate that $\mathcal{I}_T(\mu) := \infty$ if the marginal of $\mu$ over its
first two variables is not equal to $\kappa$ (recall that $\kappa$ is the measure
over $D \times \mathbb{R}^q$ that the empirical measure at time $0$ converges to). Otherwise for $\mu \in X_T$, $x \in D$ and $z\in \mathbb{R}^q$ let $\mu_{x,z} \in
\mathcal{P}\big(  C([0,T],\mathbb{R}^q) \big)$ denote the law of
$\mu$ conditioned on the values of its first two variables. We can now define
\begin{align}\label{eq: I T mu definition}
\mathcal{I}_T(\mu) :=  \int_{D \times \mathbb{R}^q} \mathcal{R}\big(\mu_{x,z} || \mathcal{W}_T(x)\big) \kappa(dx,dz),
\end{align}
where $\mathcal{R}(\blank || \blank)$ is the Relative Entropy \cite{Budhiraja2019}. Note that we are following the `$0\times \infty = 0$' convention in \eqref{eq: I T mu definition}.
 
The following result is a well-known corollary of Sanov's Theorem \cite{Dembo1998}.
\begin{theorem} \label{Theorem Large Deviations of Uncoupled System}
Let $\mathcal{A}, \mathcal{O} \subseteq X_T $ be (respectively) closed and open. Then
\begin{align}
\lsup{n} n^{-1} \log \mathbb{P} \big(\tilde{\mu}^n_T  \in \mathcal{A} \big) \leq & - \inf_{\mu \in \mathcal{A}} \mathcal{I}_T(\mu) \label{eq: LDP upper bound A uncoupled}  \\
\linf{n} n^{-1} \log \mathbb{P} \big(\tilde{\mu}^n_T  \in \mathcal{O} \big) \geq & - \inf_{\mu \in \mathcal{O}} \mathcal{I}_T(\mu) .
\end{align}
Furthermore $\mu \mapsto \mathcal{I}_T(\mu)$ is lower semi-continuous, and has compact level sets.
\end{theorem}

\subsection{Large Deviations of the Averaged System}\label{subsection approximate heterogeneous averaged}

The main result of this section is that the system with averaged non-random
connectivity is a tight approximation to the system with random connectivity, on
an exponential scale. The proof proceeds by pushing-forward the empirical measure generated by the Brownian Motions by a continuous mapping. This technique was employed in the seminal work of Tanaka \cite{tanaka1984limit}, and more recently has been employed for a spatially-distributed system in \cite{maclaurin2024large}. We introduce the $n$-dimensional system with averaged
connectivity,
\begin{equation} \label{eq: v processes}
  \begin{aligned}
    & dv^j_{\alpha,t} = \bigg( - \sum_{\beta =1}^{q}  L_{\alpha\beta} v^j_{\beta,t} 
    &\begin{aligned}[t]
      &+ n^{-1} \sum_{k=1}^n \sum_{\beta =1}^{q} 
        \mathcal{K}_{\alpha\beta}(x^j_n,x^k_n) f_{\beta}( v^k_t) 
      + I_{\alpha,t}(x^j_n) \bigg) dt \\
      &+ \sum_{\beta = 1}^qG_{\alpha\beta,t}(x^j_n) dW^j_{\beta,t}, 
      \end{aligned}\\
    & v^i_\alpha  = z_\alpha^i
  \end{aligned}
\end{equation}
which shares with the particle system \cref{eq:particleComponent} identical Brownian
motions and initial conditions. We then define the empirical measure for the system
with averaged interactions
\begin{align}
\grave{\mu}^n_T = n^{-1}\sum_{j=1}^n \delta_{x^j_n , v^j} \in Y_T.
\end{align}
with $Y_T$ given by \cref{eq:YTSpace}. In the Supplementary Materials (\cref{sec:supplementary}), we prove the following lemma.
\begin{lemma} \label{Lemma Continuous Transformation}
There exists a continuous function $\Phi_T: X_T \mapsto Y_T$ such that, with unit probability,
\begin{align}
\grave{\mu}^n_T = \Phi_T(\tilde{\mu}^n_T).
\end{align}
\end{lemma}
Next, define the rate function $\mathcal{J}_T: Y_T \to \mathbb{R}$ to be such that
\begin{align}\label{eq: transformed rate function}
\mathcal{J}_T(\mu) := \inf\big\lbrace \mathcal{I}_T(\nu) \; : \;  \mu = \Phi_T(\nu) \big\rbrace .
\end{align}
This leads us to a Large Deviation Principle for the averaged system.
\begin{lemma} \label{Theorem Large Deviations Coupled Averaged}
For any sets $\mathcal{A},\mathcal{O} \subset Y_T$, with
$\mathcal{A}$ closed and $\mathcal{O}$ open, 
\begin{align}
\lsup{n} n^{-1}\log \mathbb{P}\big( \grave{\mu}^n_T  \in \mathcal{A} \big)  &\leq - \inf_{\mu \in \mathcal{A}} \mathcal{J}_T(\mu) \label{eq: intermediate ldp upper 2 }\\
\linf{n} n^{-1}\log \mathbb{P}\big( \grave{\mu}^n_T  \in \mathcal{O} \big)  &\geq - \inf_{\mu \in \mathcal{O}} \mathcal{J}_T(\mu) .\label{eq: intermediate ldp lower 2 }
\end{align}
Furthermore $\mu \mapsto \mathcal{J}_T$ is lower semi-continuous, and $\mathcal{J}_T$ has compact level sets.
\end{lemma}
\begin{proof}
Since
\begin{align}\label{eq: Phi T lifts}
\grave{\mu}^n_T = \Phi_T\big( \tilde{\mu}^n_T \big),
\end{align}
and $\Phi_T$ is continuous, the Large Deviation Principle must hold with rate function $\mathcal{J}_T$, thanks to the Contraction Principle \cite{Dembo1998}.
\end{proof}

\subsection{Disordered Coupling} \label{Section Disoredered Coupling}
In this section, we use the Large Deviations result of the system with averaged interactions (as proved in Lemma \ref{Theorem Large Deviations Coupled Averaged}) to characterize the Large Deviations of the system with disordered interactions. We start by restating Theorem \ref{Theorem LDP Coupled System}.
\begin{theorem} \label{Theorem Large Deviations Equivalence}
For any sets $\mathcal{A},\mathcal{O} \subset Y_T$, such that
$\mathcal{A}$ is closed and $\mathcal{O}$ is open, 
\begin{align}
\lsup{n} n^{-1}\log \mathbb{P}\big( \hat{\mu}^n_T  \in \mathcal{A} \big)  &\leq - \inf_{\mu \in \mathcal{A}} \mathcal{J}_T(\mu) \label{eq: intermediate ldp upper }\\
\linf{n} n^{-1}\log \mathbb{P}\big( \hat{\mu}^n_T  \in \mathcal{O} \big)  &\geq - \inf_{\mu \in \mathcal{O}} \mathcal{J}_T(\mu) .\label{eq: intermediate ldp lower }
\end{align}
\end{theorem}

Let $P^n \in \mathcal{P}\big( \mathcal{C}([0,T],\mathbb{R}^{qn}) \big)$ be the law of the stochastic processes $(v^j_t)_{j\in \mathbb{N}_n, t\in [0,T]}$ and let $Q^n \in \mathcal{P}\big( \mathcal{C}([0,T],\mathbb{R}^{qn}) \big)$ be the law of the original particle system $(u^j_t)_{j\in \mathbb{N}_n, t\in [0,T]}$. Write
\begin{equation}
z^j_{\alpha,t} = n^{-1}\sum_{\beta \in \NSet_q}\sum_{k \in \NSet_n} \big( \phi_n^{-1} K_{\alpha\beta}^{jk} 
      - \mathcal{K}_{\alpha\beta}(x^j_n,x^k_n) \big) f_{\beta}( u^k_{t}).
\end{equation}
Thanks to Girsanov's Theorem \cite{le2016brownian}
\begin{align} \label{eq: Girsanov Equation}
\frac{dQ^n}{dP^n} = \exp\big(n\Gamma_n(u) \big),
\end{align}
where $\Gamma_n: \mathcal{C}([0,T],\mathbb{R}^{qn}) \to \mathbb{R}$ is such that, writing $H(t,x) := \big(H_{\alpha\beta}(t,x) \big)_{\alpha,\beta \in \mathbb{N}_n} \in \mathbb{R}^{q \times q}$ to be the matrix inverse of $G(t,x) G(t,x)^T$,
\begin{multline}
\Gamma_n(u) = n^{-1}\sum_{j\in \mathbb{N}_n} \sum_{\alpha,\beta \in \mathbb{N}_q} \int_0^T \bigg\lbrace z^j_{\beta,s} H_{\alpha\beta}(s,x^j_n) \bigg( du^j_{\alpha,s} -  \sum_{\gamma \in \mathbb{N}_q} L_{\alpha\gamma}u^j_{\gamma,s}ds\\ - n^{-1}\phi_n^{-1}\sum_{k\in \mathbb{N}_n}\sum_{\gamma \in \mathbb{N}_q} K^{jk}_{\alpha\gamma}f(u^k_{\gamma,s}) ds \bigg) - \frac{1}{2} z^j_{\beta,s} z^j_{\alpha,s}H_{\alpha\beta}(s,x^j_n) \bigg\rbrace ds .
\end{multline}
We first show that $\Gamma_n$ is small with very high probability.
\begin{lemma}\label{Lemma Gamma n bound}
For any $\epsilon > 0$,
\begin{align}
\lsup{n} n^{-1}\log \mathbb{P}\big( \big| \Gamma_n(u) \big| \geq \epsilon \big) = - \infty.
\end{align}
\end{lemma}
\begin{proof}
It suffices to prove that for any $\ell , \epsilon > 0$,  
\begin{equation} \label{eq: Girsanov to show}
  \lsup{n} n^{-1}\log \mathbb{P}\big( |\Gamma_n(u)| \geq \epsilon   \big) \leq - \ell.  
\end{equation}
Through rearranging, we find that $Q^n$-almost-surely,
\begin{align} \label{eq: Gamma n u expansion}
\Gamma_n(u) =\frac{1}{2n} \sum_{j \in \mathbb{N}_n} \sum_{\alpha,\beta \in \mathbb{N}_q} \int_0^T  z^j_{\beta,s} z^j_{\alpha,s}H_{\alpha\beta}(s,x^j_n) ds +  X_T,
\end{align}
where
\[
X_t = n^{-1}\sum_{j \in \mathbb{N}_n} \sum_{\alpha,\beta,\gamma \in \mathbb{N}_q} \int_0^t z^j_{\alpha,s}H_{\alpha\beta}(s,x^j_n) G_{\beta\gamma}(s,x^j_n) dW^j_{\gamma,s}.
\]
Now our assumptions on the diffusion coefficient dictate that the singular values of $G_{\alpha\beta}(s,x^j_n)$ possess (i) a uniform upper bound and (ii) a strictly positive lower bound. This implies that the operator norm of $H(s, x^j_n )$ is uniformly upperbounded by some constant $\bar{C}$, for all $s\leq T$, all $j\in \NSet_n$ and all $n\geq 1$). We thus find that there is a constant $c > 0$ such that
\begin{align}
\frac{1}{2n} \bigg| \sum_{j \in \mathbb{N}_n} \sum_{\alpha,\beta \in \mathbb{N}_q} \int_0^T  z^j_{\beta,s} z^j_{\alpha,s}H_{\alpha\beta}(s,x^j_n) ds \bigg| \leq \frac{c}{2n} \sum_{j \in \mathbb{N}_n} \sum_{\alpha  \in \mathbb{N}_q} \int_0^T  (z^j_{\alpha,s})^2 ds .
\end{align}
Our assumption on the connectivity in Hypothesis \ref{hyp:connections} implies that
\begin{align}
 \frac{c}{2n} \sum_{j \in \mathbb{N}_n} \sum_{\alpha  \in \mathbb{N}_q} \int_0^T  (z^j_{\alpha,s})^2 ds \to 0
\end{align}
as $n\to\infty$, at a uniform rate.

It thus remains to prove that for arbitrary $\epsilon > 0$,
\begin{equation} \label{eq: Girsanov to show 2}
  \lsup{n} n^{-1}\log \mathbb{P}\big( | X_T| \geq \epsilon  \big) \leq - \ell. 
\end{equation}
Now $X_t$ is a Martingale, and our assumptions imply that its quadratic variation possesses a uniform upperbound of the form, for some constant $C > 0$ ($C$ is independent of $n$),
\begin{align}
qv(t) \leq C n^{-2} \sum_{j\in \mathbb{N}_n} \sum_{\alpha \in \mathbb{N}_q} \int_0^t (z^j_{\alpha,s})^2 ds.
\end{align}
Thus, thanks to Hypothesis \ref{hyp:connections}, there must exist a non-random sequence $(\epsilon_n)_{n\geq 1}$ such that (i) $\lim_{n\to\infty} \epsilon_n = 0$ and (ii)
\begin{align} \label{eq: QV goes to zero}
n   qv(T) \leq \epsilon_n .
\end{align}
Using the fact that a continuous martingale can be represented as a time-rescaled Brownian Motion $w(t)$ (see \cite{karatzasBrownianMotionStochastic1998}),
\begin{equation} \label{eq: Girsanov to show 3}
  \mathbb{P}\big( |X_T| \geq \epsilon  \big) \leq \mathbb{P}\big( \sup_{t \leq n^{-1}\epsilon_n} \big| w( t  )\big| \geq \epsilon \big).
\end{equation}
Since $\epsilon_n \to 0$, standard properties of Brownian Motion \cite{morters2010brownian} dictate that 
\[
 \lsup{n} n^{-1} \log \mathbb{P}\big( \sup_{t \leq n^{-1}\epsilon_n} \big| w( t  )\big| \geq \epsilon \big)  = -\infty,
\]
as required.
\end{proof}
We are now ready to prove Theorem \ref{Theorem Large Deviations Equivalence}.
\begin{proof}
We start with the upperbound \eqref{eq: intermediate ldp upper }. Define the event
\[
\mathcal{V}^n_{\delta} = \big\lbrace |\Gamma_n(u)| \leq \delta \big\rbrace .
\]
Let $\mathcal{O} \subset Y_T$ be open. Then. 
\begin{align*}
    \linf{n} n^{-1}\log \mathbb{P}\big( \hat{\mu}^n_T  \in \mathcal{O} \big) &\geq \linf{n} n^{-1}\log \mathbb{P}\big( \hat{\mu}^n_T  \in \mathcal{O} , \mathcal{V}^n_{\delta} \big) \\
    &\geq -\delta +    \linf{n} n^{-1}\log \mathbb{P}\big( \grave{\mu}^n_T  \in \mathcal{O} , \mathcal{V}^n_{\delta} \big),
\end{align*}
by using \eqref{eq: Girsanov Equation} to change variables. Now
\begin{align*}
\mathbb{P}\big( \grave{\mu}^n_T  \in \mathcal{O} , \mathcal{V}^n_{\delta} \big) =
\mathbb{P}\big( \grave{\mu}^n_T  \in \mathcal{O} \big) - \mathbb{P}\big( \grave{\mu}^n_T  \in \mathcal{O} , ( \mathcal{V}^n_{\delta})^c \big).
\end{align*}
Furthermore, it follows from Lemma \ref{Lemma Gamma n bound} that
\begin{align}
\lsup{n} n^{-1} \log \mathbb{P}\big( \grave{\mu}^n_T  \in \mathcal{O} , ( \mathcal{V}^n_{\delta})^c \big) = -\infty.
\end{align}
We thus find that
\begin{align*}
\linf{n} n^{-1} \log \mathbb{P}\big( \grave{\mu}^n_T  \in \mathcal{O} , \mathcal{V}^n_{\delta} \big) 
= \linf{n} n^{-1}\log \mathbb{P}\big( \grave{\mu}^n_T  \in \mathcal{O} \big) 
\geq - \inf_{\mu \in \mathcal{O}} \mathcal{J}_T(\mu),
\end{align*}
using the LDP lower bound in Lemma \ref{Theorem Large Deviations Coupled Averaged}. Taking $\delta \to 0$, we therefore obtain that
\[
\linf{n} n^{-1}\log \mathbb{P}\big( \hat{\mu}^n_T  \in \mathcal{O} \big) \geq - \inf_{\mu \in \mathcal{O}} \mathcal{J}_T(\mu).
\]
Turning to the upperbound, let $\mathcal{A} \subset Y_T$ be closed. Then for any $\delta > 0$,
\begin{align*}
    \lsup{n} n^{-1}\log \mathbb{P}\big( \hat{\mu}^n_T  \in \mathcal{A} \big) \leq & \max\bigg\lbrace \lsup{n} n^{-1}\log \mathbb{P}\big( \hat{\mu}^n_T  \in \mathcal{A} , \mathcal{V}^n_{\delta} \big) , \lsup{n} n^{-1}\log \mathbb{P}\big( (\mathcal{V}^n_{\delta})^c \big) \bigg\rbrace \\
    = & \lsup{n} n^{-1}\log \mathbb{P}\big( \hat{\mu}^n_T  \in \mathcal{A} , \mathcal{V}^n_{\delta} \big),
    \end{align*}
    thanks to Lemma \ref{Lemma Gamma n bound}. Thanks to Girsanov's Theorem (i.e. \eqref{eq: Girsanov Equation}),
\begin{align*}
  \lsup{n} n^{-1}\log \mathbb{P}\big( \hat{\mu}^n_T  \in \mathcal{A} , \mathcal{V}^n_{\delta} \big)   \leq &\delta +    \lsup{n} n^{-1}\log \mathbb{P}\big( \grave{\mu}^n_T  \in \mathcal{A} \big).
\end{align*}
Using the LDP upperbound of Lemma \ref{Theorem Large Deviations Coupled Averaged},
\begin{align}
\lsup{n} n^{-1}\log \mathbb{P}\big( \grave{\mu}^n_T  \in \mathcal{A} \big)  &\leq - \inf_{\mu \in \mathcal{A}} \mathcal{J}_T(\mu) ,  
\end{align}
and hence after taking $\delta \to 0^+$ we find that
\begin{align}
\lsup{n} n^{-1}\log \mathbb{P}\big( \hat{\mu}^n_T  \in \mathcal{A} \big)  &\leq - \inf_{\mu \in \mathcal{A}} \mathcal{J}_T(\mu) . 
\end{align}
\end{proof}

\subsection{Limiting Equations} \label{Section Limiting Equations}

We finish by proving the almost-sure convergence of \cref{Theorem Almost Sure Convergence}. We will do this by showing that the rate function $\mathcal{J}_T$ has a unique zero.
\begin{proof}

We first notice that $\mathcal{I}_T$ has a unique zero, written as $\nu_* \in
X_T$. $\nu_{*}$ can be written as the law of random variables
$(x,u_0,w)$ that are such that: (i) the law of $(x,u_0)$ is $\kappa$ and (ii)
the law of $w \in C\big( [0,T], \mathbb{R}^q \big)$, conditionally on the other
variables, is $\mathcal{W}_T(x)$. This fact follows from the definition in \cref{eq:
I T mu definition}, since $\mathcal{R}(\alpha || \mathcal{W}_{T}(x))$ is strictly
positive, except when $\alpha = \mathcal{W}_T(x)$, in which case it is identically
zero \cite{Budhiraja2019}. 

Write $B_{\epsilon}(\nu_*) \subset X_T$ to be the $\epsilon$-ball about $\nu_*$ (with respect to the Wasserstein Distance). The Large Deviations Upperbound \cref{eq: LDP upper bound A uncoupled} thus implies that
\begin{align}
\lsup{n} n^{-1}\log \mathbb{P} \big( \tilde{\mu}^n_T \notin B_{\epsilon}(\nu_*) \big) < 0.
\end{align}
Since $\Phi_T$ is continuous, and we have the identity \cref{eq: Phi T lifts}, we therefore find that
\begin{align}
\lsup{n} n^{-1}\log \mathbb{P}\big( \grave{\mu}^n_T \notin \Phi_T\big( B_{\epsilon}(\nu_*) \big) \big) = \lsup{n} n^{-1}\log P^n\big( \tilde{\mu}^n_T \notin B_{\epsilon}(\nu_*) \big) < 0.
\end{align}
Thus $\mathcal{J}_T$ must have a unique zero at $\Phi_T(\nu_{*})$. Thus for any $k\in \mathbb{Z}^+$,
\begin{align}
\lsup{n} n^{-1}\log \mathbb{P} \big( d_{Y_T}\big( \hat{\mu}^n_T , \Phi_T(\nu_*) \big)  \geq k^{-1} \big)  < 0,
\end{align}
 thanks to the Large Deviations upperbound of Theorem \ref{Theorem Large Deviations Equivalence}. Furthermore, since
\begin{align}
\sum_{n=1}^{\infty}   \mathbb{P}\big( d_{Y,T}\big( \hat{\mu}^n_T , \Phi_T(\nu_*) \big)  \geq k^{-1} \big)  < \infty.
\end{align}
by the Borel-Cantelli Lemma, with unit probability 
\[
\lsup{n} d_{Y,T}\big( \hat{\mu}^n_T , \Phi_T(\nu_*) \big) < k^{-1}.
\]
Since $k \in \mathbb{Z}^+$ is arbitrary, with unit probability it must be that
\[
\lim_{n\to\infty} d_{Y_T}\big( \hat{\mu}^n_T , \Phi_T(\nu_*) \big) = 0.
\]
To finish, we characterize $\Phi_T(\nu_*) := \mu$. From the definition of the mapping $\Phi_T$,
$\mu \in \mathcal{P}\big( D \times C([0,T],\mathbb{R}^q) \big)$ is the law of random variables $(x,v)$, where $(x,v_0)$ is distributed according to $\kappa$, and (conditionally on $x$), for Brownian Motions $\lbrace W_{\alpha,t} \rbrace_{1\leq \alpha \leq q}$ that are independent of the initial conditions,
\begin{multline}
dv_{\alpha,t} = \bigg( - \sum_{\beta=1}^q L_{\alpha\beta} v_{\beta,t} + I_\alpha(t,x)+ \sum_{\beta=1}^q \int \mathcal{K}_{\alpha\beta}(x,y)f_{\beta}(\tilde{v}_s)\mu(dy,d\tilde{v}) \bigg) dt\\ + \sum_{\beta=1}^q G_{\alpha\beta,t}(x) dW_{\beta,t}. \label{eq: limiting SDE }
\end{multline}
Standard theory dictates that (i) there is a unique strong solution to \cref{eq:
limiting SDE } and (ii) the solution $v_{\alpha,t}$ is Gaussian (conditionally on
$x$) \cite{karatzasBrownianMotionStochastic1998}. The equations governing the evolution of the mean and variance (as outlined in Section \ref{ssec:mainResults}) follow immediately (see for instance \cite[Section 5.6]{karatzasBrownianMotionStochastic1998}). 
\end{proof}

\section{Conclusion}\label{Section Conclusion}

We have formally derived a neural-field equation from a high-dimensional system of
neurons on a disordered network. Our aim in this paper has been to strike a balance
between biophysical accuracy (insofar as the equations are rigorously proved from a
microscopic neural network model) and mathematical tractability (we obtain Gaussian
limiting equations, which lead to limiting equations with no spatial derivatives and
a very similar structure to the classical Wilson Cowan equations
\cite{wilsonExcitatoryInhibitoryInteractions1972}). We were able to determine a range
of interesting bifurcations, including local `bump' excitations of neural activity,
and spatially structured spiral waves.

There are numerous directions for future research. We wish to more
  fully explore the dynamics of the limiting equations in Lemma \ref{Lemma Limiting
  Equations}. After the variance has relaxed to equilibrium, these equations have the
  form of the classical Wilson-Cowan equations, except with a different firing-rate
  function. One thus expects that many of the patterns that have already been found
  in neural field equations
  \cite{ermentroutSpatiotemporalPatternFormation2014,coombesNeuralFields2014,cook2022neural} can
  also be found for the equations in Lemma \ref{Lemma Limiting Equations}. A
  significant advantage of our equations is that one has a more concrete
  understandiung of how exactly they arise out of a particle model.

Another promising avenue is to extend these results to include delays due to synaptic
processing / transmission. One expects this to yield fundamentally different limiting
equations, since we have $O(n \varphi_n)$ extra synaptic variables.

Finally, for parameterizations involving two or more attractors, we will leverage the
Large Deviations result to compute the most likely transition paths between
attractors. As we briefly surveyed in the Introduction, it is widely conjectured in
the Theoretical Neuroscience community that noise-induced transitions in large
ensembles of neurons occur regularly in the brain (in scenarios such as binocular
rivalry \cite{moreno2007noise} and UP /DOWN transitions). For even moderately large ensembles of neurons $n
> O(100)$ a direct computation of (say) an expected transition time from one
attractor to another is completely computationally intractable. Since the Large
Deviations rate function gives an asymptotic estimate for the probability of a
particular transition pathway, it provides a computationally-tractable means of
estimating the likelihood of a transition. Let us note that using the Large
Deviations theory to estimate most likely transition pathway will likely still be
very computationally demanding, since it requires one to minimize an
infinite-dimensional function. However the key difference is that the Large
Deviations computational complexity does not diverge with $n$.

\appendix

\section{Proof of \cref{lem:contraction}}\label{sec:appendixProofContraction}
\begin{proof} Fix $x \in D$, set $V_x(t) = V(x,t)$, $G_x = G(x)G^T(x)$, and rewrite
  the second and fourth equation in \cref{eq:meanField} as the following ODE on
  $\RSet^{q \times q}$
\[
  \frac{d}{dt} V_x = -(L V_x + V_x L^T) + G_x G_x^T, \qquad V_x(0) = V_{x,0},
\]
whose equilibria satisfy a Sylvester equation. The dynamical system can be
transplanted on $\RSet^{q^2}$ using the vectorisation operator, setting $v_x = \vect V_x$,
$g_x = \vect (G_x G_x^T)$, and $M = I_q \otimes L + L^T
\otimes I_q$, where $I_q$ is the identity in $\RSet^{q \times q}$, leading to
\begin{equation}\label{eq:vxEquation}
  \frac{d}{dt} v_x = -M v_x + g_x, \qquad v_x(0) = v_{x,0}.
\end{equation}
We note that 
\[
\sigma(M) = \{ \lambda + \rho : \lambda, \rho \in \sigma(L)  \}, \qquad 
\max_{\lambda \in \sigma(M)} \real \lambda < - 2\gamma,
\]
hence from standard ODE theory it follows that: (i) There exists a unique equilibrium
$v_x^*$ to \cref{eq:vxEquation}. (ii) The unique solution to \cref{eq:vxEquation} is
given by
    \[
      v_x(t) = e^{-tM}v_{x,0} + \int_{0}^{t}e^{-(t-s)M} g_x \,ds.
    \]
    (iii) By \cite[Definition 3.12--Theorem 3.14]{engel2000a}
    there exist $\beta, \epsi >0$ such that
    \[
      \| v_x(t) - v_x^* \|_{\RSet^{q^2}} \leq \beta e^{-\epsi t} \| v_{x,0} - v_x^*
      \|_{\RSet^{q^2}}.
    \]
The result now follows from recalling that $\| V_x(t) \|_{F} = \| v_x(t)
\|_{\RSet^{q^2}}$ and noting that, since the mappings $x \mapsto G_x$ and $x \mapsto
V_0(x)$ are in $C(D,\RSet^{q \times q})$, then so is $x \mapsto v_x(t)$ for any $t
\in \RSet_{\geq 0}$.
\end{proof}

\section{Supplementary Material} \label{sec:supplementary}
In the Supplementary Materials we prove some additional results, and include further
numerical simulations. Recall the $n$-dimensional system with averaged
connectivity,
\begin{equation} \label{eq: v processes supplementary}
  \begin{aligned}
    & dv^j_{\alpha,t} = \bigg( - \sum_{\beta =1}^{q}  L_{\alpha\beta} v^j_{\beta,t} 
    &\begin{aligned}[t]
      &+ n^{-1} \sum_{k=1}^n \sum_{\beta =1}^{q} 
        \mathcal{K}_{\alpha\beta}(x_n^j,x_n^k) f_{\beta}( v^k_t) 
      + I_{\alpha,t}(x_n^j) \bigg) dt \\
      &+ \sum_{\beta = 1}^qG_{\alpha\beta,t}(x_n^j) dW^j_{\beta,t}, 
      \end{aligned}\\
    & v^i_\alpha  = z_\alpha^i
  \end{aligned}
\end{equation}
which shares with the original particle system identical Brownian
motions and initial conditions. Our first result concerns the push-forward function $\Phi_T$ referred to in Section 6.

\subsection{Transformation of the Uncoupled System to the Coupled System}
\label{Section Transformation}

Consider the empirical measure generated by \cref{eq: v processes supplementary}, that is, the
particle system with average coupling
\[
  \grave{\mu}^n_T = n^{-1}\sum_{j=1}^n \delta_{x^j_n , v^j} .
\]
Our first main aim is to prove Lemma 6.2. We must thus find a continuous mapping $\Phi_T: X_T \to Y_T$,  such that $\grave{\mu}^n_T =
\Phi_T\big( \tilde{\mu}^n_T \big)$ identically, with $\tilde \mu^n_T$ given by
\begin{equation}\label{eq:muTildeDefSupp}
  \tilde{\mu}^n_T = n^{-1} \sum_{j \in \NSet_n} \delta_{b^j}
  \quad 
  b^j = \bigl(x_n^j,u_0^j,\{ \tilde W^j_{t} \colon t \in [0,T] \}\bigr),
  \quad
  \tilde W^j_{\alpha,t} = \sum_{\beta=1}^q \int_0^t G_{\alpha\beta,s}(x_n^j) dW_{\beta,s}^j
\end{equation}

 This allows us to `push-forward'
the Large Deviations Principle for the uncoupled system (as noted in 
Section 6.1 of the main paper) to obtain a Large Deviations
Principle for the coupled system. To the knowledge of these authors, the first
scholar to apply this technique to the Large Deviations of interacting particle
systems was Tanaka \cite{tanaka1984limit}. One of these authors has used this
technique to determine the Large Deviations of a spatially-distributed network of
interacting neurons in \cite{maclaurin2024large}.

The mapping $\Phi_T$ is defined in two steps, using an intermediate mapping $\psi_T$
that we are now going to discuss. The mapping $\psi_T$ can be thought of as as
transforming the characteristics of the noise empirical measure $\tilde{\mu}^n_T$ to
the characteristics of $\grave{\mu}^n_T$. 

Define
\[
   \begin{aligned}
     \psi_T : Y_T \times D \times \mathbb{R}^q \times C\big( [ 0,T], \mathbb{R}^q \big) 
       & \to  D \times C\big( [ 0,T], \mathbb{R}^q \big) \\
       (\nu,x,z,w) 
       & \mapsto (x,v),
   \end{aligned}
\]
where $v \in  C\big( [ 0,T], \mathbb{R}^q \big)$ is defined to be such that for all
$(\alpha,t) \in \NSet_q  \times  [0,T]$
\[
  v_{\alpha,t} = z_{\alpha} + \int_0^t  \bigg( - 
    \sum_{\beta \in \NSet_q}  L_{\alpha\beta} v_{\beta,s} 
  + \sum_{\beta \in \NSet_q} \mathbb{E}^{(y,u)\sim \nu}\big[ \mathcal{K}_{\alpha\beta}(x,y) f_{\beta}(u_s) \big]+ I_{\alpha}(s,x) \bigg) ds + w_{\alpha,t}.
\]
 \begin{lemma}
The transformation $\psi_T $ is well-defined. Furthermore $\psi_T$ is continuous.
\end{lemma}
\begin{proof}

  To prove well-definedness, fix $(\nu,x,z, w) \in Y_{T} \times D \times \mathbb{R}^q
  \times C( [ 0,T], \mathbb{R}^q)$, and consider the map $\Lambda \colon C( [ 0,T],
  \mathbb{R}^q ) \to C( [ 0,T],
  \mathbb{R}^q )$ defined by
\begin{multline*}
  \Lambda(r)_{\alpha,t} := z_{\alpha} +  \int_0^t  \bigg( - \sum_{\beta \in \NSet_q}
  L_{\alpha\beta} r_{\beta,s} + \sum_{\beta \in  \in \NSet_q} \int \mathcal{K}_{\alpha\beta}(x,y) f_{\beta}(u_s)\nu(dy,du) \\ 
+ I_{\alpha}(s,x) \bigg) ds + w_{\alpha,t}.
\end{multline*}
We claim that $\Lambda$ has a unique fixed point $v$. If this holds, then the fixed
point $v$
satisfies $\psi_T(\nu,x,z,w) = (x,v)$, which means $\psi_T$ is a
well-defined operator on $ Y_T \times D \times \mathbb{R}^q \times C( [ 0,T],
\mathbb{R}^q)$ to $\RSet^q  \times C([0,T],\RSet^q)$.
To prove that $\Lambda$ has a unique fixed point, we introduce the norm 
\[
  \| v \|_{\rho} = \sup_{t \in [0,T]} e^{-\rho t} \| v_t \|_{\RSet^q},
\]
which is equivalent to $\| \blank \|_{T}$ on $C([0,T],\RSet^q)$ for any
$\rho >0$, because $e^{-\rho T}\| v \|_T  \leq \| v \|_{\rho}  \leq \| v \|_{T}$. We
prove that $\Lambda$ is a contraction on
$( C( [ 0,\tau], \mathbb{R}^q), \| \blank \|_\rho )$, and hence on
$( C( [ 0,\tau], \mathbb{R}^q), \| \blank \|_T )$, for any $\rho > \| L \|$ where
the latter is the operator norm of $L$, seen as an operator on $\RSet^q$ to $\RSet^q$. For any
$p,r \in ( C( [ 0,\tau], \mathbb{R}^q ), \| \blank \|_\rho )$, we estimate 
\[
  \begin{aligned}
    \| \Lambda(r) - \Lambda(p)\|_\rho  
    & \leq \sup_{t \in [0,T]} e^{-\rho t} \int_{0}^{t} e^{\rho s} 
      \| e^{-\rho s} L(r_s - p_s) \|_{\RSet^q}\,d s \\
    & \leq \| L \| \; \| r - p \|_{\rho} \sup_{t \in [0,T]} e^{-\rho t}  
      \int_{0}^{t} e^{\rho s}\,d s \\
    & = \| L \| \; \| r - p \|_{\rho} \sup_{t \in [0,T]} \frac{1-e^{-\rho t}}{\rho}
     \leq  \frac{\| L \|}{\rho} \| r - p \|_{\rho}
  \end{aligned}
\]
which proves that $\Lambda$ is a contraction on $( C( [ 0,T], \mathbb{R}^q), \|
\blank \|_\rho )$, for any $\rho > \|  L \|$. This proves $\Lambda$ has a unique
fixed point in $( C( [ 0,T], \mathbb{R}^q ), \| \blank \|_T)$.

Notice that the function $(y,v) \to \mathcal{K}_{\alpha\beta}(x,y)f_{\beta}(v)$
is Lipschitz and bounded, for all $x \in D$. We thus find that there must exist a universal constant $C_2$ such that for all 
$\mu,\tilde{\mu} \in Y_{t}$, for any $x\in D$,
\begin{multline}\label{eq: Lipschitz convolution}
\bigg| \int_0^t\sum_{\beta \in \NSet^q} \int \mathcal{K}_{\alpha\beta}(x,y) f_{\beta}(u_s)\mu(dy,du) - \\ \int_0^t\sum_{\beta \in \NSet^q} \int \mathcal{K}_{\alpha\beta}(x,y) f_{\beta}(u_s)\tilde{\mu}(dy,du) \bigg|   \leq C_2 d_{Y_t}(\mu,\tilde{\mu})
\end{multline}
Since the drift term in is uniformly Lipschitz in $v$, standard techniques imply that $\psi_T$ is continuous.

\end{proof}

We next define $\Phi_T:  X_{T} \to Y_T$ to be such that
\begin{align}\label{eq: Phi T definition}
\Phi_T(\mu) := \mu \circ \psi_T( \Phi_T(\mu) ,\cdot,\cdot,\cdot)^{-1}.
\end{align}
In words, $\Phi_T(\mu)$ is the push-forward of $\mu$ by the characteristic map $\psi_T$. 
The transformation $\Phi_T$ is useful thanks to the following result.
\begin{lemma} \label{Lemma Phi T lifts}
With unit probability,
\begin{align}\label{eq: Phi T lifts in Lemma}
\grave{\mu}^{n}_T =\Phi_T\big( \tilde{\mu}^{n}_T \big),
\end{align}
\end{lemma}
\begin{proof}
This is almost immediate from the definitions.
\end{proof}

\begin{lemma}
For any $\mu \in X_T$ there exists a unique solution $\Phi_T(\mu) \in X_T$ to the fixed point identity \cref{eq: Phi T definition}. Furthermore the associated mapping $\mu \mapsto \Phi_T(\mu)$ is continuous over $X_{T}$.
\end{lemma}
\begin{proof}
Fix $\mu \in X_{T}$. 
For any $\gamma \in Y_{T}$, define $\Gamma_{\mu,T}(\gamma) := \mu \circ \psi_T( \gamma ,\cdot,\cdot,\cdot)^{-1} \in Y_T$. 
One easily checks that for any $\gamma,\tilde{\gamma} \in Y_{T}$, there is a constant $C > 0$ such that
\begin{align}
d_{Y_t}\big( \Gamma_{\mu,t}(\gamma) , \Gamma_{\mu,t}(\tilde{\gamma}) \big) \leq Ct d_{Y_t}\big( \gamma , \tilde{\gamma} \big).
\end{align}
Hence for small enough $t$, $\Gamma_{\mu,t}$ is a contraction and there is a unique fixed point solution $\mu_{*,t}$ to \cref{eq: Phi T definition}. 
 
Next, define the space $\tilde{Y}_t$ to consist of all measures $\nu \in Y_T$ such that the law of the variables upto time $t$ is identical to $\mu_{*,t}$. For $\gamma,\tilde{\gamma} \in \tilde{Y}_t$  we find that for $s \geq t$,
\begin{align}
d_{Y_s}\big( \Gamma_s(\gamma) , \Gamma_s(\tilde{\gamma}) \big) \leq C(s-t) d_{Y_s}\big( \gamma , \tilde{\gamma} \big).
\end{align}
Furthermore the constant $C$ depends on $T$ only. Hence $\Gamma$ is a contraction for small enough $s-t$, and we obtain a unique fixed point in $\tilde{Y}_t$. Iterating this argument, we obtain a unique fixed point $\Phi_T(\mu)$ upto time $T$.

For the continuity, let $\mu,\tilde{\mu} \in X_{T}$, and let $\xi_{\epsilon}$ be a measure that is within $\epsilon \ll 1$ of realizing the infimum in the definition of the Wasserstein metric. That is, we write $\xi_{\epsilon}$ to be the law of coupled random variables $(x,u_0,w) , (\tilde{x},\tilde{u}_0,\tilde{w}) \in D \times \mathbb{R}^q \times C([0,T],\mathbb{R}^q) $, and $\xi_{\epsilon}$ is such that
\begin{align}
\mathbb{E}^{\xi_{\epsilon}}\big[ \| x- \tilde{x} \| + \| u_0 - \tilde{u}_0 \| + \| w - \tilde{w} \|_T \big] \leq \epsilon.
\end{align} 
Write $\psi_T\big( \Phi_T(\mu) , x , u_0 ,  w \big) := (x,u_{[0,T]})$ and $\psi_T\big( \Phi_T(\tilde{\mu}) , \tilde{x} , \tilde{u}_0 , \tilde{w} \big) := (\tilde{x},\tilde{u}_{[0,T]})$. Substituting definitions, and employing the Lipschitz property in \eqref{eq: Lipschitz convolution}, we find that there is a constant $C > 0$ (chosen independently of $T$) such that
\begin{multline}
\sup_{t\leq T}\sup_{1\leq \alpha \leq q}\big| u_{\alpha}(t) - \tilde{u}_{\alpha}(t) \big| \leq \sup_{1\leq \alpha \leq q}\big| u_{\alpha}(0) - \tilde{u}_{\alpha}(0) \big| \\+ C \int_0^T \big\lbrace  \sup_{t\leq T}\sup_{1\leq \alpha \leq q}\big| u_{\alpha}(t) - \tilde{u}_{\alpha}(t) \big| + d_{Y_t}\big(\Phi_t(\mu) , \Phi_t(\tilde{\mu}) \big) \\
 + \norm{ x-\tilde{x}} \big\rbrace dt + \sup_{t\leq T}\sup_{1\leq \alpha \leq q}\big| w_{\alpha}(t) - \tilde{w}_{\alpha}(t) \big| .
\end{multline}
Taking expectations of both sides with respect to $\xi_{\epsilon}$, and writing
\begin{align}
y_s = \mathbb{E}\big[ \sup_{t\leq s} \norm{ u_{\alpha}(t) - \tilde{u}_{\alpha}(t) } + \norm{x - \tilde{x}}\big],
\end{align}
we obtain that for all $t\geq 0$,
\begin{align}
y_t \leq \epsilon +  C \int_0^t\big\lbrace y_s +d_{Y_s}\big(\Phi_s(\mu) , \Phi_s(\tilde{\mu}) \big) +  d_{X_T}(\mu,\tilde{\mu}) + \epsilon  \big\rbrace ds +  d_{X_T}(\mu,\tilde{\mu}) 
\end{align}
since
\[
\mathbb{E}^{\xi_{\epsilon}} \big[ \sup_{t\leq T}\sup_{1\leq \alpha \leq q}\big| w_{\alpha}(t) - \tilde{w}_{\alpha}(t) \big| \big] \leq \epsilon + d_{X_T}(\mu,\tilde{\mu}).
\]
Note that, by definition of the Wasserstein Metric,
\[
d_{Y_s}\big(\Phi_s(\mu) , \Phi_s(\tilde{\mu}) \big) \leq y_s .
\]
Taking $\epsilon \to 0^+$, we thus find that
\[
d_{Y_t}\big(\Phi_t(\mu) , \Phi_t(\tilde{\mu}) \big) \leq C\int_0^t \big( 2 d_{Y_s}\big(\Phi_s(\mu) , \Phi_s(\tilde{\mu}) \big) + d_{X_T}(\mu,\tilde{\mu}) \big) ds + d_{X_T}(\mu,\tilde{\mu}).
\]
An application of Gronwall's Inequality then implies that there exists a constant $\tilde{C}_T$ such that
\begin{align}
d_{Y_T}\big(\Phi_T(\mu) , \Phi_T(\tilde{\mu}) \big) \leq \tilde{C}_T d_{X_T}(\mu,\tilde{\mu}).
\end{align}
Thus $\Phi_T$ is Lipschitz (and also continuous).
\end{proof}

%
%

\subsection{Bounding the Original Particle System}

Our main results only concern the convergence of empirical averages of bounded
continuous functions. In fact it is possible to show that the empirical average of
certain unbounded continuous functions also converges. This is particularly desirable
for our Gaussian Application, because Gaussian distributions are most conveniently
described in terms of their first and second moments. The main result of this section
is the following Corollary to Theorem 3.9.

\begin{corollary}
Let $h: \mathbb{R}^q \mapsto \mathbb{R}$ be continuous and such that
\[
 \| h(z) \| \leq \rm{Const} \| z \|^2.
\]
Then for any $t\leq T$, any continuous function $g: D \mapsto \mathbb{R}$, it holds that
\begin{align}
\lim_{n\mapsto \infty} \bigg| n^{-1} \sum_{j\in \mathbb{N}_n}g(x_n^j) h(u^j_t) - \mathbb{E}^{(x,u) \sim \bar{\mu}_t}\big[ g(x) h(u) \big] \bigg| = 0.
\end{align}
\end{corollary}
\begin{proof}
For any $c > 0$, define $h_c: \mathbb{R}^q \mapsto \mathbb{R}$ to be such that   
\begin{align}
h_{c}(u) = h( c u / \norm{u} ) \text{ in the case that }\norm{u} \geq c
\end{align}
and if $\norm{u} < c$, define $h_c(u) = h(u)$. Notice that $h_c$ is continuous and bounded, and also that
\begin{equation} \label{eq: Gaussian h c tail}
\lim_{c\to\infty} \sup_{t\leq T} \mathbb{E}^{(x,u) \sim \bar{\mu}_t}\big[ g(x) \big( h(u) - h_c(u) \big)\big]  = 0.
\end{equation}
Since $h_c$ is continuous and bounded, Theorem 1 implies that
\begin{align} \label{eq: convergence h c }
\lim_{n\to \infty} \bigg| n^{-1} \sum_{j\in I_n}g(x_n^j) h_c(u^j_t) - \mathbb{E}^{(x,u) \sim \bar{\mu}_t}\big[ g(x) h_c(u) \big] \bigg| = 0.
\end{align}
We therefore wish to show that
\begin{equation} \label{eq: intermediate second moment corollary}
\lim_{c\to\infty} \lim_{n\to \infty} \bigg| n^{-1} \sum_{j\in I_n}g(x^j_n) \big( h(u^j_t) - h_c(u^j_t) \big) \bigg| = 0.
\end{equation}
Indeed \eqref{eq: Gaussian h c tail}, \eqref{eq: convergence h c } and \eqref{eq: intermediate second moment corollary} suffice for the Corollary, because they imply that 
 \begin{multline}
\lim_{n\to \infty} \bigg| n^{-1} \sum_{j\in I_n}g(x^j_n) h(u^j_t) - \mathbb{E}^{(x,u) \sim \bar{\mu}_t}\big[ g(x) h(u) \big] \bigg| = \\
 \lim_{c\to \infty} \lim_{n\to  \infty} \bigg| n^{-1} \sum_{j\in I_n}g(x^j_n) \big( h(u^j_t) - h_c(u^j_t) \big)\\ - \mathbb{E}^{(x,u) \sim \bar{\mu}_t}\big[ g(x) \big( h(u) - h_c(u) \big)\big] \bigg| = 0.
\end{multline}
It only remains to prove \eqref{eq: intermediate second moment corollary}, and in fact this is a consequence of Lemma \ref{Lemma tails for the second moment convergence} below.
\end{proof}
We next bound the tails of the second moment of the original particle system. 
\begin{lemma} \label{Lemma tails for the second moment convergence}
For any $\epsilon > 0$, there exists $c_{\epsilon} > 0$ such that 
\begin{align}
\lsup{n} n^{-1} \log \mathbb{P}\bigg( n^{-1} \sum_{j\in I_n} \| u^j_t \|^2 \chi\lbrace \| u^j_t \| \geq c_{\epsilon} \rbrace \geq \epsilon \bigg) < 0.
\end{align}
\end{lemma}
\begin{proof}
Write the matrix exponential as $Q_t = \exp\big( -tL \big)$. The solution of \eqref{eq:particleModel restated}, written in its mild form, satisfies the identity
\begin{align}
u^j_t = Q_t u^j_0 +\int_0^t Q_{t-s}  \Bigl(    \frac{1}{n} \sum_{k=1}^{n} \frac{1}{\phi_n}K^{jk} f(u^k_s) + I^j_s \Bigr) ds + \int_0^t  Q_{t-s} G^j_s dW^j_s .
\end{align}
Taking the norm of both sides, squaring, and then using the inequality $(a_1 + a_2 + a_3 + a_4)^2 \leq 4 a_1^2 + 4a_2^2 + 4a_3^2 + 4 a_4^2$, we find that
\begin{multline}
\| u^j_t \|^2 \leq 4 \| Q_t u^j_0 \|^2 + 4\left\| \int_0^t Q_{t-s} I^j_s ds \right\|^2 +  \\ 4 \left\| \int_0^t \frac{1}{n \phi_n} \sum_{k=1}^{n} Q_{t-s}  K^{jk} f(u^k_s) ds \right\|^2 + 4 \left\|  \int_0^t  Q_{t-s} G^j_s dW^j_s \right\|^2.
\end{multline}
Using Jensen's Inequality, and the fact that $|f| \leq f_{max}$, there is a constant such that
\begin{align}
 \left\| \int_0^t \frac{1}{n \phi_n} \sum_{k=1}^{n} Q_{t-s}  K^{jk} f(u^k_s) ds \right\|^2  \leq t \rm{Const}   \int_0^t\bigg( \frac{1}{n \phi_n} \sum_{k=1}^{n}   K^{jk}\bigg)^2  ds  .
\end{align}
Summing over $j$, and employing Hypothesis 3.5, we find that as long as $c_{\epsilon}$ is sufficiently large, it must hold that for all $t\leq T$,
\begin{align}
4 n^{-1}\sum_{j\in I_n}  \left\| \int_0^t \frac{1}{n \phi_n} \sum_{k=1}^{n} Q_{t-s}  K^{jk} f(u^j_s) ds \right\|^2 \leq \frac{1}{4} c_{\epsilon}.
 \end{align}
 Since it is assumed that the initial empirical measure converges, as long as $c_{\epsilon}$ is sufficiently large, 
 \begin{align}
\lsup{n} n^{-1}\sum_{j\in I_n}  4 \| Q_t u^j_0 \|^2 \leq \frac{1}{4} c_{\epsilon}.
 \end{align}
Now the inputs are assumed to be uniformly bounded, i.e.
  \[
  \sup_{j \in I_n} \sup_{s \leq T}  \| I^j_s \| < \infty,
  \]  
and hence as long as $c_{\epsilon}$ is sufficiently large, it must hold that
 \begin{align}
\lsup{n} n^{-1}\sum_{j\in I_n}  4 \left\| \int_0^t Q_{t-s} I^j_s ds \right\|^2  \leq \frac{1}{4} c_{\epsilon}.
 \end{align}
 For the stochastic integral, by Ito's Lemma,
 \begin{multline}
n^{-1}\sum_{j\in I_n}  \left\|  \int_0^t  Q_{t-s} G^j_s dW^j_s \right\|^2 = n^{-1}\sum_{j\in I_n} \int_0^t \rm{tr}\big( Q_{t-s} G^j_s ( G^j_s)^T Q_{t-s}^T \big) ds \\ + 2 n^{-1}\sum_{j\in I_n} \int_0^t (X^j_s)^T Q_{t-s} G^j_s dW^j_s
 \end{multline}
 where we have written
 \[
X^j_s =   \int_0^t  Q_{t-s} G^j_s dW^j_s \in \mathbb{R}^d.
 \]
 Since $\sup_{j\in I_n} \sup_{s \leq T} \| G^j_s \| < \infty$, it holds that there is a constant such that for all $n\geq 1$ and all $t\leq T$,
 \[
 n^{-1}\sum_{j\in I_n} \int_0^t \rm{tr}\big( Q_{t-s} G^j_s ( G^j_s)^T Q_{t-s}^T \big) ds < \rm{Const}.
 \]
 One can then show that for any $\delta > 0$,  
 \begin{align}
\sup_{n\geq 1}n^{-1} \log \mathbb{P}\bigg( 2 n^{-1}\sum_{j\in I_n} \int_0^t (X^j_s)^T Q_{t-s} G^j_s dW^j_s \geq \delta \bigg) < 0.
 \end{align}
We thus find that as long as $c_{\epsilon}$ is large enough
\begin{align}
\lsup{n} n^{-1} \log \mathbb{P} \bigg( n^{-1} \sum_{j\in I_n} 4 \left\|  \int_0^t  Q_{t-s} G^j_s dW^j_s \right\|^2 > \frac{1}{4} c_{\epsilon} \bigg) < 0.
\end{align}
Combining the above results, we can thus conclude that
\begin{align}
\lsup{n} n^{-1} \log \mathbb{P}\bigg( n^{-1} \sum_{j\in I_n} \| u^j_t \|^2 \chi\lbrace u^j_t \| \geq c_{\epsilon} \rbrace \geq \epsilon \bigg) < 0.
\end{align}
\end{proof}

\subsection{Proof that the Connectivity Assumptions are satisfied for a sparse
Erdos-Renyi Random Digraph} \label{Section Connectivity Satisfied for Erdos Renyi
Random Graphs}

We prove that the connectivity assumptions in Hypothesis 3.5 are satisfied in the case that the connections take on values in $\lbrace -1,0,1\rbrace$ and are sampled independently from a probability distribution. More precisely, we take $\phi_n$ to be a positive sequence that decreases to zero, and such that there exists a constant $l > 0$ such that
\begin{equation} \label{eq: summability of phi n}
\sum_{n=1}^{\infty} \exp\big( - l n\phi_n \big) < \infty.
\end{equation}
It is also assumed that there exist continuous functions $p^+_{\alpha\beta}, p^-_{\alpha\beta} : D \times D \to \mathbb{R}^+$ such that
\begin{align}
\mathbb{P}\big( K^{jk}_{\alpha\beta} = 1 \big) &= \phi_n p^+_{\alpha\beta}(x^j_n,x^k_n)\label{eq: probability of positive connection} \\
\mathbb{P}\big( K^{jk}_{\alpha\beta} = -1 \big) &= \phi_n p^-_{\alpha\beta}(x^j_n,x^k_n).\label{eq: probability of negative connection}
\end{align}
Furthermore it is assumed that $K^{jk}_{\alpha\beta}$ is probabilistically independent of $K^{ab}_{\gamma\delta}$ if either $a\neq j$, and /or $b \neq k$. This implies that the connectivity will not be symmetric (in general). We define the averaged connectivity to be such that
\begin{align}
\mathcal{K}_{\alpha\beta}(x,y) = p^+_{\alpha\beta}(x,y) - p^-_{\alpha\beta}(x,y).
\end{align}

We start by proving the following Lemma.
\begin{lemma} \label{Lemma Bound Number of nonzero connections in row}
For any $c \in \lbrace -1,1\rbrace$,
    \begin{equation}\label{eq: absolute summability of connectivity sparseness appendix}
\lsup{n}(n\phi_n)^{-1} \sup_{\alpha\in \mathbb{N}_q}\sup_{j\in \mathbb{N}_n}  \sum_{k\in \mathbb{N}_n}\sum_{\beta\in\mathbb{N}_q} \chi\lbrace K^{jk}_{\alpha\beta} =c \rbrace < \infty.
    \end{equation}
    \end{lemma}
\begin{proof}
Thanks to Chernoff's Inequality, for any $L > 0$, $c \in \lbrace -1,1 \rbrace$ and any 
$j\in \mathbb{N}_n$, $\alpha \in \mathbb{N}_q$, for any $a > 0$,
\begin{align*}
\mathbb{P}\bigg( \sum_{k\in \mathbb{N}_n}\sum_{\beta\in\mathbb{N}_q} \chi\lbrace K^{jk}_{\alpha\beta} =c \rbrace \geq L n\phi_n  \bigg) \leq
 \mathbb{E}\bigg[ \exp\bigg( a\sum_{k\in \mathbb{N}_n}\sum_{\beta\in\mathbb{N}_q} \chi\lbrace K^{jk}_{\alpha\beta} =c \rbrace -a L n \phi_n \bigg) \bigg].
\end{align*}
Our assumptions \eqref{eq: probability of positive connection}-\eqref{eq: probability of negative connection} dictate that there is a constant $C> 0$ such that for all $j,k\in \mathbb{N}_n$ and $\alpha \in \mathbb{N}_q$, 
\begin{align}
 \mathbb{E}\bigg[ \exp\bigg( a \sum_{\beta\in\mathbb{N}_q} \chi\lbrace K^{jk}_{\alpha\beta} =c \rbrace \bigg) \bigg] \leq 1 + \phi_n C\big( \exp(qa) - 1 \big).
\end{align}
We thus find that
\begin{multline}
\mathbb{E}\bigg[ \exp\bigg( a\sum_{k\in \mathbb{N}_n}\sum_{\beta\in\mathbb{N}_q} \chi\lbrace K^{jk}_{\alpha\beta} =c \rbrace -a L n \phi_n \bigg) \bigg] \\
\leq \exp\bigg( n\phi_n C\big( \exp(qa) - 1 \big) - a L n \phi_n \bigg).
\end{multline}
Thus for large enough $L$,
\begin{align}
\mathbb{P}\bigg( \sum_{k\in \mathbb{N}_n}\sum_{\beta\in\mathbb{N}_q} \chi\lbrace K^{jk}_{\alpha\beta} =c \rbrace \geq L n\phi_n  \bigg) \leq \exp(-ln \phi_n).
\end{align}
Thus, thanks to assumption \eqref{eq: summability of phi n},
\[
\sum_{n=1}^{\infty} \mathbb{P}\bigg( \sum_{k\in \mathbb{N}_n}\sum_{\beta\in\mathbb{N}_q} \chi\lbrace K^{jk}_{\alpha\beta} =c \rbrace \geq L n\phi_n  \bigg) < \infty .
\]
An application of the Borel-Cantelli Lemma then implies that
    \begin{equation}
\lsup{n}(n\phi_n)^{-1} \sup_{\alpha\in \mathbb{N}_q}\sup_{j\in \mathbb{N}_n}  \sum_{k\in \mathbb{N}_n}\sum_{\beta\in\mathbb{N}_q} \chi\lbrace K^{jk}_{\alpha\beta} =c \rbrace \leq L.
    \end{equation}
\end{proof}
We next obtain a bound on the operator norm of the connectivity matrix. For a constant $c > 0$, define the event
\begin{multline}
\mathcal{Q}_n = \bigg\lbrace   \sup_{w \in \mathbb{R}^{nq} : \| w \| = 1} \sum_{j,k \in \mathbb{N}_n} \sum_{\alpha,\beta \in \mathbb{N}_q}\chi\big\lbrace K^{jk}_{\alpha\beta} = 1 \big\rbrace w^j_{\alpha} w^k_{\beta} \leq c n\phi_n \text{ and }\\   \sup_{w \in \mathbb{R}^{nq} : \| w \| = 1} \sum_{j,k \in \mathbb{N}_n} \sum_{\alpha,\beta \in \mathbb{N}_q}\chi\big\lbrace K^{jk}_{\alpha\beta} = -1 \big\rbrace w^j_{\alpha} w^k_{\beta} \leq c n\phi_n  \bigg\rbrace .
 \end{multline}
 We notice that if the event $\mathcal{Q}_n$ holds, then necessarily
 \begin{align} \label{eq: bound operator norm last section}
 \sup_{w \in \mathbb{R}^{nq} : \| w \| = 1} \sum_{j,k \in \mathbb{N}_n} \sum_{\alpha,\beta \in \mathbb{N}_q}  K^{jk}_{\alpha\beta}   w^j_{\alpha} w^k_{\beta} \leq 2 c n \phi_n.
 \end{align}
 \begin{lemma}
There exists a constant $c > 0$ such that, with unit probability there exists a random integer $n_0$ such that for all $n \geq n_0$, the event $\mathcal{Q}_n$ holds.
 \end{lemma}
 \begin{proof}
Thanks to the Perron-Frobenius Thoerem, for $a \in \lbrace -1,1 \rbrace$,
\begin{align}
\sup_{w \in \mathbb{R}^{nq} : \| w \| = 1} \sum_{j,k \in \mathbb{N}_n} \sum_{\alpha,\beta \in \mathbb{N}_q}\chi\big\lbrace K^{jk}_{\alpha\beta} = a \big\rbrace w^j_{\alpha} w^k_{\beta} \leq \sup_{\alpha\in \mathbb{N}_q}\sup_{j\in \mathbb{N}_n}  \sum_{k\in \mathbb{N}_n}\sum_{\beta\in\mathbb{N}_q} \chi\lbrace K^{jk}_{\alpha\beta} =a \rbrace .
\end{align}
The Lemma is thus a consequence of Lemma \ref{Lemma Bound Number of nonzero connections in row}.
 \end{proof}
The next lemma concerns the limits of the constants $R^n$ of Hypothesis 3.5. 
\begin{lemma}
With unit probability,
\begin{align}
\lim_{n\to\infty}R^n = 0.
\end{align}
\end{lemma}
\begin{proof}
Define
\begin{align*}
&R^n_+ = n^{-1}\sup_{\alpha \in \mathbb{N}_q}\sup_{y\in [-1,1]^{qn} }\sum_{j \in \mathbb{N}_n , \alpha \in \mathbb{N}_q} &&\bigg(  n^{-1}\sum_{k\in \mathbb{N}_n , \beta \in \mathbb{N}_q}\bigg( \phi_n^{-1} \chi\big\lbrace K^{jk}_{\alpha\beta} = 1 \big\rbrace \\ & &&  -   p^+_{\alpha\beta}(x^j_n , x^k_n) \bigg) y^k_{\beta} \bigg)^2 \\
&R^n_- = n^{-1}\sup_{\alpha \in \mathbb{N}_q}\sup_{y\in [-1,1]^{qn} }\sum_{j \in \mathbb{N}_n , \alpha \in \mathbb{N}_q} &&\bigg(  n^{-1}\sum_{k\in \mathbb{N}_n , \beta \in \mathbb{N}_q}\bigg( \phi_n^{-1} \chi\big\lbrace K^{jk}_{\alpha\beta} = -1 \big\rbrace \\ & &&  -   p^-_{\alpha\beta}(x^j_n , x^k_n) \bigg) y^k_{\beta} \bigg)^2
\end{align*}
It suffices that we show that
\begin{align}
\lim_{n\to\infty} R^n_- =& 0 \label{eq: p minus last part in lemma}  \\
\lim_{n\to\infty} R^n_+ =& 0. \label{eq: p plus last part in lemma}  
\end{align}
Since the proofs are almost identical, we only prove \eqref{eq: p plus last part in lemma}. 
For $y\in [-1,1]^{qn}$, define
\[
  h^j_{\alpha}(y) = n^{-1} \bigg| \sum_{k=1}^n \sum_{\beta=1}^q \big( \phi_n^{-1} \chi\lbrace K_{\alpha\beta}^{jk} = 1 \rbrace - p^+_{\alpha\beta}(x_n^j,x_n^k) \big) y^k_{\beta} \bigg|.
\]
We need to show that
\begin{equation} \label{eq: to show square norm h j alpha}
\lim_{n\to\infty}\bigg\lbrace n^{-1} \sup_{y\in [-1,1]^{qn}} \sup_{\alpha \in \mathbb{N}_q}\sum_{j\in I_n} h^j_{\alpha}(y)^2 \bigg\rbrace = 0.
\end{equation}
For a constant $C > 0$, define also the event
\begin{equation}
\mathcal{U}_n = \bigg\lbrace \text{For }c=\pm 1, \; \; \sup_{\alpha\in \mathbb{N}_q}\sup_{j\in \mathbb{N}_n}  \sum_{k\in \mathbb{N}_n}\sum_{\beta\in\mathbb{N}_q} \chi\lbrace K^{jk}_{\alpha\beta} =c \rbrace \leq C n\phi_n \bigg\rbrace .
\end{equation}
The constant $C > 0$ is taken to be large enough that $\mathcal{U}_n$ always holds (for large enough $n$), which is possible thanks to Lemma \ref{Lemma Bound Number of nonzero connections in row}.

For $m \in \mathbb{Z}^+$, we decompose
\begin{equation}
y^j_{\alpha} = y^{(m),j}_{\alpha} + \tilde{y}^{(m),j}_{\alpha}
\end{equation}
where $\tilde{y}^{(m),j}_{\alpha} \in [0, m^{-1})$ and $y^{(m),j}_{\alpha} = am^{-1}$ for some integer $a$. Write
\begin{align}
\mathcal{S}^n_m = \big\lbrace y \in [-1,1]^{qn} \; : \; y^j_{\alpha} = a^j_{\alpha} / m \text{ for some integer }a^j_{\alpha} \big\rbrace.
\end{align}
Notice that
\begin{equation}
\lsup{n} n^{-1} \log \big| \mathcal{S}^n_m \big| < \infty. 
\end{equation}
We write
\begin{align}
 \sum_{j\in \mathbb{N}_n} \big( h^j_{\alpha}(y)\big)^2 =&  \sum_{j\in \mathbb{N}_n}\big\lbrace  h^{j}_{\alpha}(y^{(m)})^2 + h_{\alpha}^j(\tilde{y}^{(m)})^2 + 2h^j_{\alpha}(y^{(m)})h^j_{\alpha}(\tilde{y}^{(m)}) \big\rbrace \nonumber \\
 \leq &  \sum_{j\in \mathbb{N}_n}\big\lbrace  h^{j}_{\alpha}(y^{(m)})^2 + h_{\alpha}^j(\tilde{y}^{(m)})^2\big\rbrace\nonumber \\
 &+ 2\bigg\lbrace \sum_{j\in \mathbb{N}_n} h^{j}_{\alpha}(y^{(m)})^2 \bigg\rbrace^{1/2} \bigg\lbrace  \sum_{j\in \mathbb{N}_n}  \tilde{h}^{j}_{\alpha}(y^{(m)})^2 \bigg\rbrace^{1/2},\label{eq: cauchy-schwarz inequality at end}
\end{align}
thanks to the Cauchy-Schwarz Inequality. Furthermore one verifies that (as long as the event $\mathcal{U}_n$ holds), for all $j\in \mathbb{N}_n$ and all $\alpha\in\mathbb{N}_q$,
\begin{align}
 \big| h_{\alpha}^j(\tilde{y}^{(m)}) \big| \leq C m^{-1} +qm^{-1} \sup_{x,y\in \mathcal{D}, \beta \in \mathbb{N}_q}  p^{+}_{\alpha\beta}(x,y) 
 := \tilde{C} m^{-1},
 \end{align}
 and the RHS evidently goes to $0$ uniformly as $m\to\infty$.  In light of \eqref{eq: cauchy-schwarz inequality at end}, in order that \eqref{eq: to show square norm h j alpha} holds it thus suffices that we show that
 \begin{align}
  \lim_{m\to\infty}   \lim_{n\to\infty} n^{-1} \sup_{\alpha \in \mathbb{N}_q} \sup_{y\in \mathcal{S}^n_m}\sum_{j\in \mathbb{N}_n}h^j_{\alpha}(y)^2 = 0. \label{eq: to show lim m lapha n }
 \end{align}
 Now for any $\delta > 0$, using a union-of-events bound,
 \begin{align}
\mathbb{P}\big( \sup_{y\in \mathcal{S}^n_m}  \sum_{j\in \mathbb{N}_n}h^j_{\alpha}(y)^2 \geq n\delta , \mathcal{U}_n \big) \leq & \big| \mathcal{S}^n_m \big| \sup_{y\in \mathcal{S}^n_m} \mathbb{P}\bigg( \mathcal{U}_n,  \sum_{j\in \mathbb{N}_n}h^j_{\alpha}(y)^2 \geq n\delta \bigg) \nonumber \\
=& (m+1)^n \sup_{y\in \mathcal{S}^n_m} \mathbb{P}\bigg( \mathcal{U}_n,    \sum_{j\in \mathbb{N}_n}h^j_{\alpha}(y)^2 \geq n\delta \bigg) . \label{eq: temporary intermediate bound h squared}
 \end{align}
 Thanks to the Borel-Cantelli Lemma, in order that \eqref{eq: to show lim m lapha n } holds it suffices that we show that for arbitrary $\delta > 0$,
 \begin{align}
    \lsup{n} n^{-1} \log  \sup_{y\in \mathcal{S}^n_m} \mathbb{P}\big( \mathcal{U}_n,   \sum_{j\in \mathbb{N}_n}h^j_{\alpha}(y)^2 \geq n\delta \big) = -\infty, \label{eq: to show intermediate y j alpha squared} 
 \end{align}
 since \eqref{eq: temporary intermediate bound h squared} and \eqref{eq: to show intermediate y j alpha squared} imply that
 \begin{align}
\sum_{n=1}^{\infty} \mathbb{P}\bigg( \sup_{y\in \mathcal{S}^n_m} \sum_{j\in \mathbb{N}_n}h^j_{\alpha}(y)^2 \geq n\delta , \mathcal{U}_n \bigg) < \infty.
 \end{align}
To this end, define
\begin{align}
g^j_{\epsilon}(y) = \chi\big\lbrace \sup_{\alpha\in \mathbb{N}_q} h^j_{\alpha}(y) \geq \epsilon \big\rbrace .
\end{align}
Notice that if the event $\mathcal{U}_n$ holds, then
\[
\big| h^j_{\alpha}(y) \big| \leq C + \sup_{\beta\in \mathbb{N}_q} \sup_{x,z \in D} p^+_{\alpha\beta}(x,z)  := \bar{C}.
\]
This means that (as long as the event $\mathcal{U}_n$ holds), then for all $y\in [-1,1]^{nq}$,
\begin{equation}
 n^{-1}\sum_{j\in \mathbb{N}_n}h^j_{\alpha}(y)^2 \leq \bar{C}^2  n^{-1}\sum_{j\in \mathbb{N}_n} g^j_{\epsilon}(y) + \epsilon^2 .
\end{equation}
In order that \eqref{eq: to show intermediate y j alpha squared} holds, it thus suffices that we show that for arbitrary $\epsilon > 0$,
\begin{align}
\lim_{n\to\infty} n^{-1}\log \sup_{y\in \mathcal{S}^n_m} \mathbb{P}\big( n^{-1} \sum_{j\in \mathbb{N}_n} g^j_{\epsilon}(y) \geq \epsilon \big) = -\infty.    
\end{align}
To this end, by Chernoff's Inequality, for a constant $a > 0$,
\begin{align}
\mathbb{P}\big( n^{-1} \sum_{j\in \mathbb{N}_n} g^j_{\epsilon}(y) \geq \epsilon \big) \leq & \mathbb{E}\bigg[ \exp\bigg( a\sum_{j\in \mathbb{N}_n} g^j_{\epsilon}(y) - a \epsilon n \bigg) \bigg] \nonumber \\
=& \exp\big(-a \epsilon n \big) \prod_{j\in \mathbb{N}_n}\bigg(1 + \mathbb{P}\big( g^j_{\epsilon}(y) = 1 \big) \big( \exp(a) - 1 \big) \bigg) \nonumber\\
\leq & \exp\bigg( n p_n \big( \exp(a) -1 \big) - a \epsilon n \bigg) 
\end{align}
where 
\[
p_n = \sup_{y\in \mathcal{S}^n_m} \sup_{j\in \mathbb{N}_n} \mathbb{P}\big( g_{\epsilon}^j(y) =1 \big).
\]
We define $a = - \log p_n$, and it remains for us to prove that
\begin{align}
\lim_{n\to\infty} p_n = 0. \label{eq: p n limit}
\end{align}
In fact \eqref{eq: p n limit} follows almost immediately from the Hoeffding Inequality \cite{Massart2019}.
\end{proof}

%
%

\subsection{Example of noise-induced Turing-like bifurcation}
\label{ssec:turingExampleOld} 
\begin{figure}
  \centering
  \includegraphics{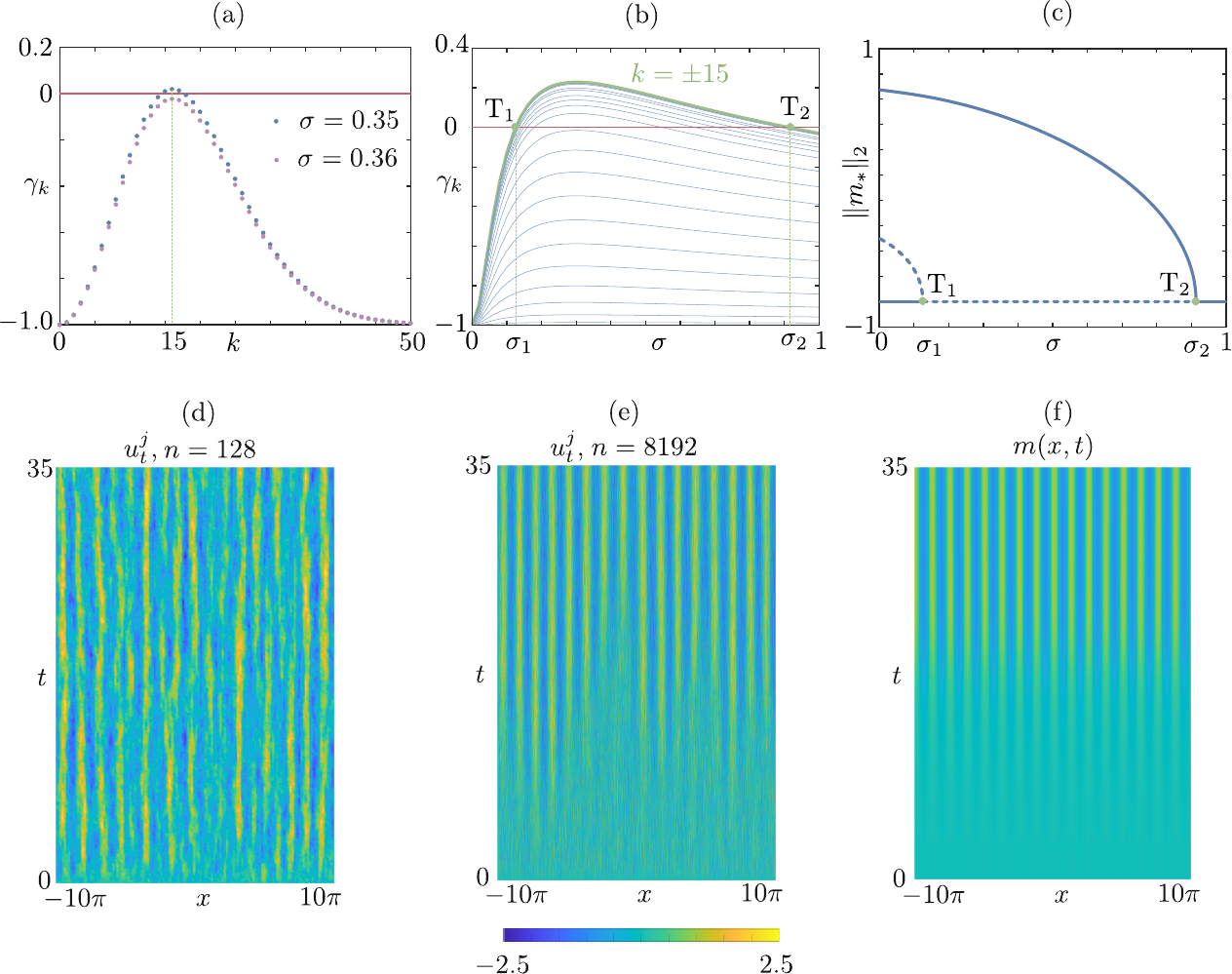}
  \caption{Noise-induced Turing-like bifurcation for the particle and mean-field
    model with one population ($q=1$), posed on a ring of width $2l$, with synaptic
    kernel \cref{eq:ADefOld}, neuronal firing rate \cref{eq:particleFiringRateOld}, and
    mean-field firing rate \cref{eq:firingRateMeanFieldOld}. (a): values $\{ \gamma_k
    \}_k$ (defined as in the main text of the manuscript),
    show that a Turing-like bifurcation of the homogeneous steady state is located
    between $\sigma = 0.35$ and $\sigma = 0.36$, with critical wavenumber $k_c =
  15$. (b) Curves $\gamma_k(\sigma)$ for $k = \{0, \pm 1, \ldots, \pm 20 \}$ show
  that the branch of homogeneous steady states is stable when there is no noise
  ($\sigma = 0$) and undergoes two noise-induced Turing-like bifurcations at $\sigma =
  \sigma_{1,2}$, both for $k_c = 15$. (c): numerical bifurcation analysis for spatially-extended
  equilibria of the mean-field shows that the first bifurcation is subcritical,
  while the second one is supercritical (dashed lines represent unstable branches).
Numerical simulations of the particle system with kernel matrix for $n = 128$
(e), $n = 8192$ (f), and the mean-field for $\sigma = 0.58 \in
(\sigma_1,\sigma_2)$ confirm that the bifurcation structure in (c). Parameters:
$L =1$, $l =10 \pi$, $I(t,x) \equiv 0$, $G(x,t) \equiv \sigma$, $B = 1.5$, $C = 7$,
$\alpha = 10$, $\theta = 0.4$, $m_0(x) = 0.3\cos(k_c \pi x/l)$. 
}
  \label{fig:turing}
\end{figure}

\rev{
We give a further example of Turing-like bifurcation, in a model with a kernel that
does not support localised solutions, as shown in the main text. We consider a network with $q=1$ on
a ring of width $2l$, $D = \RSet/2l\ZSet$ with distance dependent kernel $\calK(x,y)
= A(x-y)$, where 
\begin{equation}\label{eq:ADefOld}
  A(x) = \frac{C}{\sqrt{\pi}} e^{-x^2} - \frac{C}{B \sqrt{\pi}} e^{-(x/B)^2}, \qquad
  B \in \RSet_{>1}, \qquad C \in \RSet_{>0},
\end{equation}
linear coupling $L =1$, forcing $I(t,x) \equiv 0$, $G(t,x) \equiv \sigma$, and with
firing rate function 
\begin{equation}\label{eq:particleFiringRateOld}
  f(u) = \Phi(\alpha(u-\theta)), \qquad \Phi(u) = \frac{1}{2}\biggl[1 + \erf\biggl(\frac{u}{\sqrt{2}}\biggr)\biggr],
  \qquad 
  \alpha \in \RSet_{>0},
  \quad
  \theta \in \RSet,
\end{equation}  
which results in a mean-field firing rate of the form
\cite{touboulNoiseInducedBehaviorsNeural2012a}
\begin{equation}\label{eq:firingRateMeanFieldOld}
  F(m,v) = \Phi\biggl(\alpha \frac{m-\theta}{\sqrt{1+\alpha^2 v}}\biggr).
\end{equation}

The synaptic kernel $A$ is locally excitatory and laterally inhibitory, and has been
selected here as it is \textit{balanced}, in the sense that its integral over $\RSet$
is null. If $D = \RSet$, it can be shown that $m_*=0$ is a homogeneous equilibrium for any value of $\sigma$. We
consider the problem on a ring of width $2l$, hence
\[
  \int_{-l}^{l} A(x) \,d x = C ( \erf{l} - \erf{(l/B)} ),
\]
which, for large $l$ is approximately null. We have computed a branch of
$\sigma$-dependent homogeneous steady states using Newton's scheme, and the computed
values of $m_*$ do not appreciably differ from $0$ in the selected parameter range. 

In \cref{fig:turing}(a) we show $\{ \gamma_k \}_{k \in \NSet_{50}}$ for two
values of $\sigma$, which provides evidence of a bifurcation at $\sigma_c \in
(0.35,0.36)$ with wavenumber $k_c = 15$. In \cref{fig:turing}(b) we computed the
curves $\gamma_k(\sigma)$ for $k = \{ 0,\pm 1, \ldots, \pm 20 \}$, from which we
deduce that the homogeneous steady state is stable in the noiseless case $\sigma = 0$,
and it undergoes two Turing-like bifurcations induced by noise. 

We employed numerical bifurcation analysis tools to study the
bifurcation structure of steady states to of the mean-field equation (see
\cref{fig:turing}(c)). We see that the primary bifurcation $T_1$ is subcritical, and
a branch of stable spatially-periodic steady states emerge. In contrast, the
secondary bifurcation, $T_2$, is supercritical and gives rise to stable
spatially-periodic equilibria. 

We carried out time simulations of the particle system with varying numbers of
neurons, and of the mean-field equation to confirm the predictions of the numerical
bifurcation analysis in \cref{fig:turing}(d)--(f). 
}

\bibliographystyle{siamplain}
\bibliography{neuralfieldsbib}
\end{document}